\documentclass[12pt,a4paper]{amsart}
\makeatletter
\renewcommand\normalsize{%
	\@setfontsize\normalsize{11.7}{14pt plus .3pt minus .3pt}%
	\abovedisplayskip 10\p@ \@plus4\p@ \@minus4\p@
	\abovedisplayshortskip 6\p@ \@plus2\p@
	\belowdisplayshortskip 6\p@ \@plus2\p@
	\belowdisplayskip \abovedisplayskip}
\renewcommand\small{%
	\@setfontsize\small{9.5}{12\p@ plus .2\p@ minus .2\p@}%
	\abovedisplayskip 8.5\p@ \@plus4\p@ \@minus1\p@
	\belowdisplayskip \abovedisplayskip
	\abovedisplayshortskip \abovedisplayskip
	\belowdisplayshortskip \abovedisplayskip}
\renewcommand\footnotesize{%
	\@setfontsize\footnotesize{8.5}{9.25\p@ plus .1pt minus .1pt}
	\abovedisplayskip 6\p@ \@plus4\p@ \@minus1\p@
	\belowdisplayskip \abovedisplayskip
	\abovedisplayshortskip \abovedisplayskip
	\belowdisplayshortskip \abovedisplayskip}
\ifdefined\pdfpagewidth
\else
\fi
\calclayout
\makeatother
\usepackage{caption}
\usepackage{float}
\usepackage{tikz}
\usepackage{amstext,amsmath,amssymb,amsfonts}
\usepackage[utf8]{inputenc}
\usepackage[T1]{fontenc}
\usepackage[english]{babel}
\usepackage{stmaryrd}
\usepackage{epsfig}
\usepackage{hyperref}
\usepackage{color}

\usepackage{geometry}
\usepackage{amsthm}

\definecolor{myurlcolor}{rgb}{0,0,0.7}

\newcommand{\N}{\mathbb{N}}

\newcommand{\R}{\mathbb{R}}

\newtheorem{thm}{Theorem}[section]
\newtheorem{coro}{Corollary}[section]
\newtheorem{lem}{Lemma}[section]
\newtheorem{con}{Cons\'equence}[section]
\newtheorem{pro}{Proposition}[section]
\newtheorem{dfn}{Definition}[section]
\newtheorem{rmq}{Remark}[section]
\numberwithin{equation}{section}
\newtheorem{expl}{Example}[section]
\newtheorem{com}{Commentary}[section]
\newcommand{\bprof}{\begin{prof}}
	\newcommand{\eprof}{\end{prof}}
\newenvironment{prof}[1][Proof]{\textbf{#1.} }{\ \rule{0.5em}{0.5em}}
\newtheorem{prop}[thm]{\begin{Large} {\fontfamily{pzc} \Large  \textbf{ Propri\'et\'es}}\end{Large}}
\newtheorem{exos}[thm]{Exercice}

\newtheorem{sol}{Solution}
\newtheorem{apl}{Application}[section]

\newcommand{\beq}{\begin{eqnarray}}
\newcommand{\eeq}{\end{eqnarray}}
\newcommand{\bpro}{\begin{pro}}
	\newcommand{\epro}{\end{pro}}
\newcommand{\bapl}{\begin{apl}}
	\newcommand{\eapl}{\end{apl}}
\newcommand{\bprop}{\begin{prop}}
	\newcommand{\eprop}{\end{prop}}
\newcommand{\blem}{\begin{lem}}
	\newcommand{\elem}{\end{lem}}
\newcommand{\bdfn}{\begin{dfn}}
	\newcommand{\edfn}{\end{dfn}}
\newcommand{\bcom}{\begin{com}}
	\newcommand{\ecom}{\end{com}}
\newcommand{\bcor}{\begin{cor}}
	\newcommand{\ecor}{\end{cor}}
\newcommand{\bcoro}{\begin{coro}}
	\newcommand{\ecoro}{\end{coro}}
\newcommand{\bthm}{\begin{thm}}
	\newcommand{\ethm}{\end{thm}}
\newcommand{\bex}{\begin{expl}}
	\newcommand{\eex}{\end{expl}}
\newcommand{\brmq}{\begin{rmq}}
	\newcommand{\ermq}{\end{rmq}}
\newcommand{\bexos}{\begin{exos}}
	\newcommand{\eexos}{\end{exos}}
\newcommand{\bsol}{\begin{sol}}
	\newcommand{\esol}{\end{sol}}
\newcommand{\bcon}{\begin{con}}
	\newcommand{\econ}{\end{con}}
\newcommand{\benum}{\begin{enumerate}}
	\newcommand{\bnots}{\begin{not}}
		\newcommand{\enots}{\end{not}}
	\newcommand{\eenum}{\end{enumerate}}
\newcommand{\bitem}{\begin{itemize}}
	\newcommand{\eitem}{\end{itemize}}

\newcommand{\bea}{\begin{eqnarray}}
\newcommand{\eea}{\end{eqnarray}}
\newcommand{\enn}{\nonumber \end{equation}}
\newcommand{\beqs}{\begin{eqnarray*}}
\newcommand{\eeqs}{\end{eqnarray*}}

\newcommand{\cT}{\mathcal{T}}

\title[ Iterative method
for a nonlinear plate bending problem]
{Numerical solution and errors analysis of iterative method
	for a nonlinear plate bending problem }

\email{a) akakpowilfred@gmail.com}
\address{Institut de Mathématiques et de Sciences Physiques,
	Universit\'e d'Abomey-Calavi (UAC), Rep. of Benin}

\email{b) khouedanou@yahoo.fr}
\address{D\'epartement de Math\'ematiques, Facultés des Sciences et Techniques,
Universit\'e d'Abomey-Calavi (UAC), Rep. of Benin}

\author{ Akakpo Amoussou Wilfried$^{(a)}$
and Hou\'edanou Koffi Wilfrid$^{(b)}$ }

\begin{document}
\renewcommand{\contentsname}{Contents}
\maketitle
\begin{normalsize}
\begin{abstract}\normalsize
This paper uses the HCT finite element method and mesh adaptation technology to solve the nonlinear plate bending problem and conducts error analysis on the iterative method, including a priori and a posteriori error estimates.
 Our investigation exploits Hermite finite elements such as BELL and HSIEH-CLOUGH-TOCHER (HCT) triangles for conforming finite element discretization. Then, the existence and uniqueness of the approximation solution are proven by using a variant of the Brezzi-Rappaz-Raviart theorem. 
We solve the approximation problem through a fixed-point strategy and an iterative algorithm, and study the convergence  of the iterative algorithm, and provide the convergence conditions.
 An optimal a priori error estimation has been established. 
 We construct a posteriori error indicators by distinguishing between discretization and linearization errors and prove their reliability and optimality.
  A numerical test is carried out and the results obtained confirm those established theoretically. \\
	\\
	\small{\bf Mathematics Subject Classification [MSC]:} 74S05,74S10,74S15,
	74S20,74S25,74S30.
\end{abstract}
\tableofcontents
\section{Introduction}
In this paper, we are interested in a numerical solution and a priori, and a-posteriori errors analysis of iterative method
for a nonlinear plate bending problem. The numerical resolution is carried out using the HCT finite element method with mesh adaptation, allowing increased precision thanks to a local refinement based on error indicators. In indeed, the a posteriori error analysis initiated by Babuška \cite{Babuska} and improved by Verf\"{u}rth \cite{Verf} satisfies several objectives in the numerical resolution of PDEs, more precisely in the case of nonlinear PDEs. It makes it possible to globally control the discretization error of the problem posed by providing explicit bounds on the error between the numerical solution and the exact solution as soon as the approximate solution is known. This analysis can provide stopping criteria which guarantee overall error control and constitutes a basic tool for the construction of the adaptive mesh. In the context of nonlinear problems, Chaillou and Suri \cite{Chaillou1,Chaillou2} will initiate the construction of a posteriori error estimators by distinguishing between linearization and discretization errors. Then, this method will be developed within the framework of an iterative algorithm by L. El Alaoui, A. Ern \cite{Alaoui}. Another iterative method is used by C. Bernardi, Jad Dakroub, Gihane Mansour and Tony Sayah \cite{Jad} to provide us with a remarkable gain in terms of calculation time. This method is applied to the nonlinear Laplace problem.
Furthermore, an error analysis for the bilaplacian problem is addressed in the literature by P.G. Ciarlet et al. and other authors in \cite{P1,P2,P3,P4,P5,S,FB}.
Different approximation approaches are discussed after a variational formulation. We note that the variational formulation of the bilaplacian equation is simple and its conforming discretization by finite elements requires finite elements of class $C^1$ which are rather expensive due to the high polynomial degree. To overcome this difficulty, they consider a mixed discretization, or a non-conforming approximation, each having its advantages and disadvantages. Verf\"{u}rth used the three approaches with a posteriori residual error analysis in \cite{Verf}. The nonlinear case of its problems are not too addressed. We study here a nonlinear case of the plate bending problem.
{\color{red}Let $\Omega$ be a bounded open and connected of the ${\mathbb{R}}^d$ with lipschitzian boundary, $\Gamma=\partial\Omega$ the boundary of $\Omega$ , $d\in\{2,3\}$.}
We consider the plate bending nonlinear model:
\begin{equation}
\label{p}
\begin{cases}
\Delta^2 u+ \lambda|u|^{2p}u=f \text{ in } \Omega\\
u=0      \text{ on } \Gamma\\
\dfrac{\partial u}{\partial n}=0      \text{ on } \Gamma,
\end{cases} 
\end{equation}	
where $\lambda$ and $p$  are strictly positive real numbers, $f\in H^{-2}(\Omega)$ topological dual of $H^{2}_0(\Omega)$ and $n$ the unitary exterior normal at $\Gamma$. To our knowledge, there is no a posteriori error estimation result of the problem $(\ref{p})$  where a conforming finite element method is used. Here, we develop such a posteriori error analysis using isotropic (or regular) mesh. One of the main differences between our paper and the reference \cite{Jad} is that the unknown solution of the problem $(\ref{p})$ has regularity $ H_0^2(\Omega)$. This requires more regular finite elements. The standard Lagrange finite elements used in \cite{Jad} are no longer appropriate. Our investigation uses Hermite finite elements, namely BELL and HSIEH-CLOUGH-TOCHER (HCT) Triangles \cite{AR1, AR2, AR3}, which assure inclusion of our space approximation $V_h$ in $H_0^2(\Omega)$ [see \ref{vh} ].

The problem $(\ref{p})$ may arise in various contexts of mathematical physics, including solid mechanics, theory of elastic plates, differential geometry. By examining its components, it could model a physical system where an unknown $u$ is subject to forces described by the function $f$, and where deformations are subject to power-type non linearities and specific boundary conditions on the domain boundary. Thus, the equation $(\ref{p})_1$ could be applied to various scenarios where these conditions are satisfied, such as elastic plate deformations, nonlinear wave propagation phenomena. Furthermore, one may notice a similarity between this equation $(\ref{p})_1$, which are often used to describe phenomena such as liquid-liquid or solid-solid phase separation in metal alloys, grain growth in polycrystalline materials \cite{w3,w1,w2}. 
We noted that, the equation $(\ref{p})_1$ is particularly the Cahn-Hilliard equation in two dimension \cite{soglo}.

\emph{Plan of the paper.} The contents of this paper have been organized in the following manner. For the nonlinear problem (\ref{p}), we introduce a variational formulation $(\ref{fv})$ and prove the existence and uniqueness of the exact solution in section \ref{weakformulation}. In section \ref{discretisation}, we develop  a conforming discretization of the variational problem (\ref{p}) using Hermite finite elements and prove the existence and uniqueness of an approximate solution according to the Brezzi-Rappaz-Raviart theorem (cf. Theorem \ref{th}) for the approximation of nonlinear problems. Some technical results have been developed in section \ref{technique}, with fixed point strategy.
Afterwards, in section \ref{algorithme}, we use an iterative algorithm (Banach-fixed point) to make an appropriate linearization (\ref{al}) of the approximate problem (\ref{pd}), and in the section \ref{pointfixe}  we study the convergence of this algorithm towards the solution of the discrete problem. 
A priori error analysis has been deveoped in section \ref{apriori}.
 An important step is to derive a posteriori error estimates by distinguishing between linearization and discretization errors. In section \ref{aposteriori}, the a posteriori error estimates are derived. We define the error indicators (see Definition \ref{indicator}), and the upper error bound has been established in subsection \ref{upper} while the lower error bound is proved in subsection \ref{lower}. Section \ref{Numerical_results} is devoted to numerical results. Its subsection \ref{Numerical} presents the numerical solution after convergence, while subsection \ref{priori_estimation} shows the evolution of the error at each mesh size. Subsection \ref{posteriori_estimation} shows us the behavior of local error indicators with errors. An adaptation of meshes according to these errors indicators is presented in the  subsection \ref{adapt}. In the section \ref{resultats}, we analyze the performance of the developed method through numerical results, highlighting its accuracy, robustness and relevance for solving complex problems. We offer our conclusion and the further works in section \ref{Conclusion}.
\section{Weak formulation}\label{weakformulation}
We describe in this section the nonlinear problem (\ref{p}) together with its variational formulation and we proof existence and uniqueness of exact solution of the nonlinear problem.
First of all, we recall the main notion and results which we use later on. If $\mathcal{O}$ is a bounded open domain of $\mathbb{R}^d$ , we denote by $L^p(\mathcal{O})$
the space of measurable functions summable with power $p$.  For $v\in L^p(\mathcal{O})$ , 
the norm is defined by,
 \begin{equation}
\|v\|_{L^p(\mathcal{O})}=\left[\int_{\mathcal{O}}|v(x)|^pdx\right]^{\frac{1}{p}}.
\end{equation}
Now, we introduce some Sobolev spaces and norms \cite{Brezis}. 
Let $m\in \N$  and $\ d\in \N^*$ where $\mathbb{N}$ is the set of natural numbers, and $\mathbb{N}^*$ is the set of nonzero natural numbers, $1\leqslant p< \infty$, a real number. The Sobolev space $W^{m,p}(\mathcal{O})$ is defines by:
\begin{equation*}
W^{m,p}(\mathcal{O}):=\left\lbrace v\in {L^p(\mathcal{O})}:D^{\alpha}v\in{L^p(\mathcal{O})}\  \forall \alpha=(\alpha_1,\alpha_2,\ldots,\alpha_d)\in \N^d : \sum_{i=1}^{d}\alpha_i=|\alpha|\leqslant m \right\rbrace,
\end{equation*}
with the norm:
\begin{equation}
\|v\|_{W^{m,p}(\mathcal{O})}:=\left[\sum_{|\alpha|\leqslant m}\|D^{\alpha}v\|^p_{L^p(\mathcal{O})}\right]^{\frac{1}{p}} , \ 1\leqslant p<\infty\ \forall v\in W^{m,p}(\mathcal{O}).
\end{equation}
The Sobolev space $H^m(\mathcal{O})=W^{m,2}(\mathcal{O})$ is defined in the usual way with the usual norm $\parallel\cdot\parallel_{m,\mathcal{O}}$ and semi-norm $|\cdot|_{m,\mathcal{O}}$. In particular, 
$H^0(\mathcal{O})=L^2(\mathcal{O})$ and we write $\parallel\cdot\parallel_{\mathcal{O}}$ for $\parallel\cdot\parallel_{0,\mathcal{O}}$. Similarly we denote by $(\cdot,\cdot)_{\mathcal{O}}$ the $L^2(\mathcal{O})$ inner product. For shortness if $\mathcal{O}$ is equal to $\Omega$, we will drop the index $\Omega$. The space $H_0^m(\mathcal{O})$ denotes the closure of $\mathcal{C}_{c}^{\infty}(\mathcal{O})$ in $H^{m}(\mathcal{O})$ and $H^{-m}(\mathcal{O})$ is topological dual space of $H_0^m(\mathcal{O})$, equipped with the norm:
 \begin{equation}
\|z\|_{-m,\mathcal{O}}= \sup_{{u\in H_0^{m}(\mathcal{O})}-\{0\}}\dfrac{|\langle z,u\rangle|}{\|u\|_{m,\mathcal{O}}}  \quad \forall z\in H^{-m}(\mathcal{O}).
\end{equation}
{\color{red}If the open domain $\mathcal{O}$ is  bounded, connected and has a lipschitzian boundary, then for $m=2$, the map $u\mapsto \|\Delta u\|_{L^2(\mathcal{O})}$ is a norm equivalent to $|\cdot|_{2,\mathcal{O}}$ on $H_0^2(\mathcal{O})$ \cite{luq}.}
 Moreover, for all integer $m\geqslant 0$ and for all $p\in [1,+\infty[$, we have the following inclusions with continuous injections \cite[Chapter 3]{Adams}: 
$ W^{m,p}(\mathcal{O})\hookrightarrow L^{p^*}({\mathcal{O}})\text{ if } m<\tfrac{d}{p}$, 
$W^{m,p}(\mathcal{O}) \hookrightarrow L^{q}({\mathcal{O}}),  q\geq 1\text{ if } m=\tfrac{d}{p}$ and 
$ W^{m,p}(\mathcal{O}) \hookrightarrow \mathcal{C}^0(\bar{\mathcal{O}})\text{ if } m>\tfrac{d}{p}, \mbox{ with } p^*=\tfrac{pd}{d-mp}$.\\ 
{\color{red}
\textbf{Recurrent notation.}
	In the sequel, we denote by $C$, $C'$, $c_1$, $c_2$, $c_1'$, $\ldots$, generic constants that can vary from line
	to line but are always independent of all discretization parameters.}\\
We have the following lemma:
\begin{lem}(cf. \cite{Brezis})\label{ijs}
	 For all $p\in \ ]1,+\infty[$,  there exists a strictly positive real constant $C_p$ such that : 
	\begin{equation}
	\sup_{x\in \bar{\mathcal{O}}}|v(x)|\leqslant C_p\|v\|_{W^{m,p}(\mathcal{O})}, \forall v\in W^{m,p}(\mathcal{O}), \mbox{ if } m>\frac{d}{p}.
	\end{equation}
\end{lem}
The nonlinear model problem (\ref{p}) admits the equivalent variational formulation: find $u\in V:=H_0^2(\Omega)$ such that,
\begin{equation}\label{fv}
\int_{\Omega}\Delta u(x)\Delta v(x)dx+\lambda\int_{\Omega}|u|^{2p}u(x)v(x)dx=\langle f,v\rangle_{H^{-2}(\Omega)}, \forall v\in V.
\end{equation}
	Our problem being nonlinear, the Lax-Milgram theorem cannot be used to show the existence and uniqueness of solution.
	We will therefore resort to the minimization results \cite{kavian}. 
	\bthm (ref. \cite{kavian})
	\label{a}
	Let $H$ be a reflexive Banach space, $K$ a closed convex of $H$ and $J : K \to \R$ a convex function 
	lower semi-continuous (abbreviated s-c-i) if $K$ is unbounded, suppose that for any sequence $(u_n) _n$ of $K$ such that $\|u_n\|\to +\infty$, when $n \to +\infty$, we have $J(u_n)\to +\infty$. Then $J$ reaches its minimum on $K$ :
	\begin{equation}
	\exists u\in K,  J(u)=\inf_{v\in K}J(v)= \min_{v\in K}J(v).
	\end{equation}
	Moreover if $J$ is strictly convex, $u$ is unique.
	\ethm
	We introduce the following definition:
	\begin{dfn}
	The functional (a map from a vector space of functions to its scalar body) energy associated with the nonlinear model problem $(\ref{p})$ is defined by:
	\begin{equation}\label{j}
	J(u)=\dfrac{1}{2}\int_{\Omega}|\Delta u|^2dx+\dfrac{\lambda}{2p+2}\int_{\Omega}|u|^{2p}[u(x)]^2dx-\langle f,u\rangle_{H^{-2}(\Omega)},\forall u\in H_0^2(\Omega).
	\end{equation}
\end{dfn}
\begin{rmq} (Differential of $J$) Since
 $\mathcal{C}_c^{\infty}(\Omega)$ is dense in $H^{2}_0(\Omega)$ for the
	norm of $H^{2}(\Omega)$,
	then
	\begin{equation*}
	J'(u)=0\text{ in } [\mathcal{C}_c^{\infty}(\Omega)]' \mbox{ equivalently to }\Delta^2(u)+\lambda|u|^{2p}u-f =0 \text{ in } [\mathcal{C}_c^{\infty}(\Omega)]'. 
	\end{equation*}
\end{rmq}
Now, we can proof the existence and uniqueness of exact solution of the nonlinear problem (\ref{p}). We have the following result:
\begin{thm}\label{eu}
The nonlinear plate bending problem $(\ref{p})$ admits a unique solution $u\in H^{2}_0(\Omega).$
\end{thm}
\begin{proof}
	For the proof of Theorem \ref{eu}, we use Theorem \ref{a} by verifying each  assumption. The closed space
	 $V=H^{2}_0(\Omega)$ then $V$ is convex. Let us show the energy functional $J$ defines by (\ref{j})  is convex.
	 Let's pose 
	 \begin{eqnarray}
	 g_1(u)&=&\displaystyle\int_{\Omega}|\Delta u (x)|^2dx,\\
	 g_2(u)&=&\displaystyle\int_{\Omega}|u(x)|^{2p}[u(x)]^2dx,\\
	 	 g_3(u)&=&\langle f,v\rangle_{H^{-2}(\Omega)\times H_0^2(\Omega)}.
	 \end{eqnarray}
The function $g_3$ is a linear form. Therefore it is convex and $g_2$ is strictly
	convex, because the function $x\mapsto |x|^{2p+2}$ is strictly convex.
	Let $u,v \in V$ and $t\in[0,1]$ we have: $(1-t)u+tv\in H_0^2(\Omega)$ and 
	\begin{eqnarray*}
	g_1((1-t)u+tv)&=&\displaystyle\int_{\Omega}|\Delta ((1-t)u+tv)|^2dx,\\
	&\leqslant&\displaystyle\int_{\Omega}\left[(1-t)|\Delta u|+t|\Delta v|\right]^2dx.
	\end{eqnarray*}
Using the Cauchy-Schwarz inequalities, we obtain $
g_1((1-t)u+tv)\leqslant (1-t)g_1(u)+tg_1(v)
$ and $g_1$ is convex. 
We deduce from everything that $J$ is strictly convex as the sum of convex and strictly convex functions.
 Let us show that $J$ is coercive. Since 
$\displaystyle\int_{\Omega}fudx\leq\left|\displaystyle\int_{\Omega}fudx\right|\leqslant \|f\|_{-2,\Omega}\|u\|_{2,\Omega}\leqslant \frac{s}{2}\|u\|^2_{2,\Omega}+\frac{1}{2s}\|f\|^2_{-2,\Omega},$(Young inequality)
then
$-\displaystyle\int_{\Omega}fudx\geqslant -\frac{s}{2}\|u\|^2_{2,\Omega}-\frac{1}{2s}\|f\|^2_{-2,\Omega}.$
We deduce that,
\begin{equation*}
\begin{array}{ccc}
J(u)&\geqslant& \min\left(\dfrac{1}{4},\dfrac{\lambda}{2p+2}\right)\left[\| u\|^2_{2,\Omega}+\|u\|^{2p+2}_{L^{2p+2}(\Omega)}\right]-\|f\|^2_{-2,\Omega} \text{ for } s=\frac{1}{2}.
\end{array}
\end{equation*}
Hence the coercivity.
Furthermore, the energy functional being differentiable for everything  $ v\in H_0^2(\Omega)$, then it is continuous on $H_0^2(\Omega)$ and consequently s-c-i on $H_0^2(\Omega)$.
We conclude, according to Theorem \ref{a}, the nonlinear problem $(\ref{p})$ admits a unique solution in $H_0^2(\Omega)$.
\end{proof}
\section{Finite element discretization}\label{discretisation}
\subsection{Discrete problem}
Let $h>0$ be a real parameter.
The objective here is to replace the 
space $V$ by a vector subspace of finite dimension $V_h$ called the space of approximation. This is the
continuous Galerkin method.  For construction of $H_0^2(\Omega)$-conforming approximation $V_h$ [cf. (\ref{vh}) below], we use the BELL and HSIEH-CLOUGH-TOCHER elements \cite{BH} that we present in section \ref{Bell}.
The discrete problem (\ref{pd})  associated to  the plate bending nonlinear problem $(\ref{p})$ is as follows: find $u_h\in V_h$ such that,
\begin{equation}\label{pd}
\int_{\Omega}\Delta u_h(x)\Delta v_h (x)dx+\int_{\Omega}\lambda|u_h(x)|^{2p}u_h(x)v_h(x)dx=\langle f,v_h\rangle \ \forall v_h\in V_h.
\end{equation}
In order to prove the existence and uniqueness of a solution to problem (\ref{pd}), we use the Brezzi-Rappaz-Raviart theorem (cf. \cite{Brezzi}). This theorem will also allow us to control the corresponding a priori error estimator.
Before stating the theorem, we specify certain hypotheses.
Let $X$ and $Y$ be two Banach spaces.
We introduce a class $C^1$ map, $G : X\to Y$ and  continuous linear map $S\in \mathcal{L}(Y,X)$. We put for all  $u\in X,$
\begin{equation}
\label{F'}
F(u)=u-S\circ G(u).
\end{equation}
We consider a finite element approximation of a solution  $u\in X$ of the problem (\ref{F}) below:
\begin{equation}
\label{F}
F(u)=0 \text{ i.e. } u-S\circ G(u)=0.
\end{equation}
For $h>0$ a real, we give ourselves a finite-dimensional vector subspace $X_h$ of $X$ and an associated operator $S_h\in \mathcal{L}(Y, X_h)$ . The approximate problem of $(\ref{F})$ consists of finding $u_h\in X_h$ solution to the equation :
\begin{equation}
\label{Fh}
 F_h(u_h)=0, \mbox{ where } F_h(u_h)=u_h-S_h\circ G(u_h).
\end{equation}
Then, we have the following theorem (cf. \cite{Brezzi}):

\bthm\cite{Brezzi}
\label{th}
We assume that $G$ is a map of class $C^1$ from $X$ to $Y$ such that $G'$ is lipschitz continuous on the bounded subsets of $X$, i.e. there exists a function $L: \mathbb{R}_{+}\times\mathbb{R}_{+}\to \mathbb{R}_{+}$ monotonically increasing with respect to each variable such that for all $u,w\in X, \|G'(u)-G'(w)\|_{\mathcal{L}(X,Y)}\leqslant L(\|u\|_{X},\|w\|_{X})\|u-w\|_{X}$, $S\circ G'(u) \in\mathcal{L}(X)$ is compact and $F'(u)$ is an isomorphism of $X$.
Furthermore, we assume that for all $v\in X$,
\begin{equation}
\label{}
\lim\limits_{h\to 0}\|v-\Pi_hv\|_X=0,
\end{equation}
for a certain linear operator $\Pi_h\in \mathcal{L}(X,X_h)$ and
\begin{equation}
\label{}
\lim\limits_{h\to 0}\|S_h-S\|_{\mathcal{L}(Y,X)}=0.
\end{equation}
Then, there exists $h_0>0$ and an open neighborhood  $\mathcal{O}$ of the origin in $X$ such that, for each $h\leqslant h_0,$ the problem $(\ref{F})$ admits a unique solution $u_h$ satisfying $u_h-u\in \mathcal{O}$. Furthermore, we have, for a certain constant $M>0$ independent of $h$, the estimate:
\begin{equation}
\label{}
\|u_h-u\|_X\leqslant M(\|u-\Pi_h u\|_X + \|(S_h-S)\circ G(u)\|_X).
\end{equation}
\ethm
In this paper,  we apply the theorem \ref{th} to problem $(\ref{pd})$, namely (\ref{Fh}). To do this,
we need some technical results. 

\subsection{Fixed-point strategy and Some technical results}\label{technique}
\subsubsection{Fixed-point strategy}
 Let $X=H^{2}_0(\Omega)$
and  $Y=H^{-2}(\Omega)$. For the application $G : X\to Y$ and the map linear continuous $S\in \mathcal{L}(Y,X)$ of Theorem \ref{th}, we define:

\begin{eqnarray}\label{g}
G (w)=f-\lambda|w|^{2p}w, \forall w\in X
\end{eqnarray}
and 
\begin{equation}\label{s}
\begin{tabular}{cccc}
$S : $&$Y$&$\to$&$X$\\
&$\varphi$&$\mapsto$&$S(\varphi)$
\end{tabular},
\end{equation}
such that if $S(\varphi)=w$, then $w$ is the solution to the problem: 
\begin{equation}
\label{p'}
 \begin{cases}
\Delta^2 w=\varphi \text{ in } \Omega\\
w=\dfrac{\partial w}{\partial n}=0 \text{ on } \Gamma.
\end{cases}
\end{equation}
We noted that, the problem $(\ref{p'})$ admits a unique $w\in H_0^2(\Omega)$ solution to the equivalent variational problem:
\begin{equation}
\displaystyle\int_{\Omega}\Delta w(x) \Delta v(x) dx=\int_{\Omega} \varphi(x) v(x)dx, \forall v\in H_0^2(\Omega),
\end{equation}
and $S\in\mathcal{L}(Y,X)$ satisfies the inequality:
\begin{equation}
	\parallel S(\varphi)\parallel_{2,\Omega} \leq  \parallel \varphi \parallel_{-2,\Omega}. \forall \varphi \in H^{-2}(\Omega).
\end{equation}
Let introduce the following proposition:
\bpro  \label{pro1}
Let $G$ define by $(\ref{g})$ and $S$ define by $(\ref{s})$. Then, the problems $(\ref{p})$ and $(\ref{F})$ are equivalent.
\epro
\begin{proof}
	Assume that $u\in X$ is solution of $(\ref{p})$. Then, $u$ check
	$\lambda|u|^{2p}u= -\Delta^2 u+f$. 
	From the definition of $S$  and $G$, that is, $(\ref{g})$ and $(\ref{s})$,  we obtain
	$f-G(u)=0.$
	As $f-G(u)\in Y=H^{-2}(\Omega)$ and $S$ is linear, composing the above equality by $S$, we deduce 
	$S[f-G(u)]=0$ namely, $u-S\circ G(u)=0$, and 
	we have  $F(u)=0$, since $S(f)=u$.
	Reciprocally, assume that $u-S\circ G(u)=0$.
	From the definition of $S$ we have $S\circ G(u)=u$, and therefore by $(\ref{s})$, 
	$
	\Delta^2 u=G(u) \text{ in } \Omega \mbox{ with }
	u=0=\dfrac{\partial u}{\partial n} \text{ on } \Gamma.
	$
	The definition of $G$, i.e. $(\ref{g})$ conclude the result.
	
\end{proof}
\subsubsection{Some technical results}
\bpro (Technical inequality)\label{pro2}
\label{le}
Let $a,b$ and $s$ be three strictly positive real numbers. Then, we have the following inequality:
\begin{equation}
\label{IN}
|a^s-b^s|\leqslant s|a-b|\left[a^{s-1}+b^{s-1}\right].
\end{equation}
\epro
\begin{proof} 
	Let $a,b$ and $s$ be three elements of $]0,+\infty[$.
	Now, we  consider the function $\varphi\in \mathcal{C}^1([a,b],\mathbb{R})$,
	$\varphi :[a,b]\rightarrow \R$ define by
	$\varphi(x)=x^s$ . Then,  there exists  $c\in]a,b[$ constant such that $\varphi(a)-\varphi(b)=(a-b)\varphi'(c)$. 
	But $\varphi'(x)=sx^{s-1}$,  so we have
	$
	|a^s-b^s|=s|a-b||c^{s-1}|
	\leqslant s|a-b|\displaystyle\sup_{t\in [a,b]}|t^{s-1}|$. Hence
	$|a^s-b^s|\leqslant s|a-b|\left[a^{s-1}+b^{s-1}\right].$
	\end{proof}\mbox{ }\\
To be able to verify the assume of Theorem \ref{th}, we need to prove some important properties satisfy by $G$ define in $(\ref{g})$ and $S$ define in $(\ref{s})$ that will be used later.
The first lemma proves that the differential $G': H_0^2(\Omega)\to \mathcal{L}(H_0^2(\Omega),H^{-2}(\Omega)) $ is Lipschitz continuous on the bounded subsets of $H_0^2(\Omega)$.

\blem Let $G$ be the application define by (\ref{g}). Then, \\$G': H_0^2(\Omega)\to \mathcal{L}(H_0^2(\Omega),H^{-2}(\Omega))$  is Lipschitz continuous on the bounded subsets of $H_0^2(\Omega)$, where $G'$ is the differential of $G$.
\elem
\begin{proof} We can recall that $X:= H_0^2(\Omega)$ and  $Y=H^{-2}(\Omega)$. Now, let $u\in X$.
	It is well known that
	\begin{tabular}{cccc}
		$G(u)=f-\lambda|u|^{2p}u$
	\end{tabular}.
	The application $G$ is differentiable and its differential at $u\in X$ is given for $v\in X$ by:
	$G'(u)\cdot v=-\lambda(2pvu^{2p-1}u+u^{2p}v)$,  namely
	$G'(u)\cdot v=-\lambda(2p+1)|u|^{2p}v \ \forall\  v\in X.$
	To show that $G'$ is Lipschitz continuous on the bounded subsets of $H_0^2(\Omega)$, we will look for $L\in \mathbb{R}_{+}$ such that for all $(u,w)\in X^2$ we have
	$\|G'(u)-G'(w)\|_{\mathcal{L}(X,Y)}\leqslant L\|u-w\|_X.$
	Let's remember that:
	\begin{equation}
	\label{}
	\|G'(u)-G'(w)\|_{\mathcal{L}(X,Y)}=\sup_{v\in X^*}\dfrac{\|(G'(u)-G'(w))\cdot v\|_Y}{\|v\|_X},
	\end{equation} 
	and
	\begin{equation}
	\label{}
	\|(G'(u)-G'(w))\cdot v\|_Y=\sup_{z\in X^*}\dfrac{|\langle(G'(u)-G'(w))\cdot v),z\rangle|}{\|z\|_X}.
	\end{equation} 
	$\begin{array}{ccc}
	(G'(u)-G'(w))\cdot v&=&G'(u)\cdot v-G'(w)\cdot v\\\\  
	&=&-\lambda(2p+1)|u|^{2p}v+\lambda(2p+1)|w|^{2p}v\\\\
	&=&-\lambda(2p+1)(|u|^{2p}-|w|^{2p})\cdot v\\\\
	\langle (G'(u)-G'(w))\cdot v , z\rangle &=&\displaystyle \int_{\Omega}\big(-\lambda(2p+1)(|u|^{2p}-|w|^{2p})\cdot v\big) zdx\\\\
	&=& -\lambda(2p+1)\displaystyle\int_{\Omega}[(|u|^{2p}-|w|^{2p})\cdot v] zdx.         
	\end{array}$\\
	Let's calculate $\displaystyle\int_{\Omega}[(|u|^{2p}-|w|^{2p})\cdot v]zdx$.
	According to inequality \ref{IN} of Proposition \ref{pro2} and the Lemma \ref{ijs}, we have 
	$|u|^{2p}-|w|^{2p}\leqslant ||u|^{2p}-|w|^{2p}|$, namely \\
	$|u|^{2p}-|w|^{2p}\leqslant 2p|u-w|(|u|^{2p-1}+|w|^{2p-1}).$ 
	Now,
	\begin{eqnarray*}
		\displaystyle\int_{\Omega}[(|u|^{2p}-|w|^{2p})\cdot v] zdx&\leqslant&2p\int_{\Omega}|u-w|(|u|^{2p-1}+|w|^{2p-1})| v||z|dx,\\
		&\leqslant&2p\left[c_1|u-w|_{2,\Omega}(2c_2^{(2p-1)}(\|u\|^{2p-1}_{2,\Omega}+\|w\|^{2p-1}_{2,\Omega})) c_3|v|_{2,\Omega} c_4|z|_{2,\Omega}\right],\\
	\end{eqnarray*}
	which implies that,
	$$\int_{\Omega}[(|u|^{2p}-|w|^{2p})\cdot v] zdx\leqslant 4pc_1c^{(2p-1)}_2c_3c_4(\|u\|^{2p-1}_{2,\Omega}+\|w\|^{2p-1}_{2,\Omega})|u-w|_{2,\Omega}|z|_{2,\Omega}|v|_{2,\Omega}.$$
	Either,
	$$\langle G'(u)-G'(w))\cdot v), z\rangle \leqslant -4\lambda pc_1c^{(2p-1)}_2c_3c_4(\|u\|^{2p-1}_{2,\Omega}+\|w\|^{2p-1}_{2,\Omega})|u-w|_{2,\Omega}|z|_{2,\Omega}|v|_{2,\Omega},$$
	therefore,
	$$\dfrac{|\langle (G'(u)-G'(w))\cdot v, z\rangle|}{\|z\|_X}\leqslant 4\lambda pc_1c^{(2p-1)}_2c_3c_4(\|u\|^{2p-1}_{2,\Omega}+\|w\|^{2p-1}_{2,\Omega})|u-w|_{2,\Omega}|v|_{2,\Omega} \ \forall z\in X^*.$$
	By setting $L= L(\|u\|_{2,\Omega},\|w\|_{2,\Omega})=4\lambda pc_1c^{(2p-1)}_2c_3c_4(\|u\|^{2p-1}_{2,\Omega}+\|w\|^{2p-1}_{2,\Omega})\in \mathbb{R}_{+} $ we can arrive at
	$$\sup_{z\in X^*}\dfrac{|<(G'(u)-G'(w))\cdot v,z>|}{\|z\|_X}\leqslant L|u-w|_{2,\Omega}|v|_{2,\Omega},$$
	i.e
	$$\|(G'(u)-G'(w))\cdot v\|_Y\leqslant L|u-w|_{2,\Omega}|v|_{2,\Omega},$$
	which implies that
	$$\dfrac{\|(G'(u)-G'(w))\cdot v\|_Y}{\|v\|_X}\leqslant L|u-w|_{2,\Omega}.$$
	Hence,
	$$\sup_{v\in X^*}\dfrac{\|(G'(u)-G'(w))\cdot v\|_Y}{\|v\|_X}\leqslant L|u-w|_{2,\Omega}.$$
	Thus, we have 
	$\|G'(u)-G'(w)\|_{\mathcal{L}(X,Y)}\leqslant L|u-w|_{2,\Omega}.$ Since $L$ is bounded with respect to each variable $u$ and $w$ on the bounded subsets of $H_0^2(\Omega)$, then $G'$ satisfies the uniform Lipschitz condition. 
	This completes the proof.
\end{proof}
The second property prove that the  map $S\circ G'(u) : X\to X$ is linear continuous and compact on $X$, where $u\in X$.
\blem
\label{lcom}
Let $u\in X=H_0^2(\Omega)$.
The map $S\circ G'(u) : X\to X$ is linear continuous and compact on $X$, where $G$ and $S$ are  defined by (\ref{g}) and (\ref{s}) respectively.
\elem

\begin{proof}
$S$ and $G'(u)$ are two continuous linear maps on $Y$ and $X$ respectively.
Then their composite is a continuous linear map. Let us show that $S\circ G'(u)$ is compact.
Let $B_X$ be the unit ball of $X$. Then, we proof that $G'(u)(B_X)$ is relatively compact.
$B_X$ being bounded and $G'(u)$ continuous, then $G'(u)(B_X)$ is a bounded part of $Y$.
Let us show that $G'(u)(B_X)=\mathcal{H}$ is uniformly equicontinuous.\\
For $\varepsilon>0$, let's search $\alpha>0$ such that $\forall (w_1,w_2)\in X^2$, and $v\in B_X$ we have,
$\|w_1-w_2\|_X\leqslant\alpha$ implies  $|\langle G'(u)\cdot v, w_1\rangle-\langle G'(u)\cdot v, w_2\rangle|\leqslant \varepsilon.$ We have
\begin{eqnarray*}
|\langle G'(u)\cdot v, w_1\rangle-\langle G'(u)\cdot v, w_2\rangle|&=&|\langle G'(u)\cdot v, w_1-w_2\rangle|.\\
\end{eqnarray*}
It is well known that,  $G'(u)\cdot v=-\lambda(2p+1)|u|^{2p}v$.
Hence,
\begin{eqnarray*}
|\langle G'(u)\cdot v,w_1-w_2\rangle| &\leqslant& \lambda(2p+1)|\int_{\Omega}|u|^{2p}v(w_1-w_2)dx|,\\
&\leqslant& \lambda (2p+1)\int_{\Omega}||u|^{2p}v||w_1-w_2|dx,\\
&\leqslant& \lambda \mbox{mes}(\Omega)(2p+1)c_1c_2\|u\|_{L^{2p}(\Omega)}^{2p}|v|_{2,\Omega}|w_1-w_2|_{2,\Omega}.
\end{eqnarray*}
By setting $k'=\lambda mes(\Omega)(2p+1)c_1c_2\|u\|_{L^{2p}(\Omega)}^{2p}> 0$ we obtain,
$$ |\langle G'(u)\cdot v,w_1-w_2\rangle|\leqslant k'\|v\|_{X}|w_1-w_2|_{2,\Omega},$$ 
which implies,
$ |\langle G'(u)\cdot v,w_1-w_2\rangle|\leqslant k'|w_1-w_2|_{2,\Omega}$, since $v\in B_X$, i.e. 
$\parallel v\parallel_X\leq 1$.
So that $|\langle G'(u)\cdot v, w_1\rangle-\langle G'(u)\cdot v, w_2\rangle|\leqslant \varepsilon$, it suffices that $k'|w_1-w_2|_{2,\Omega}\leqslant \varepsilon$, namely,\\
$|w_1-w_2|_{2,\Omega}\leqslant\dfrac{\varepsilon}{k'}$, and
thus, we take $\alpha=\dfrac{\varepsilon}{k'}> 0$.
We deduce that $G'(u)(B_X)$ is uniformly equicontinuous.
$G'(u)(B_X)$ is bounded and uniformly equicontinuous, according to Ascoli's theorem \cite{Emily},  $G'(u)(B_X)$ is relatively compact. Therefore $G'(u)$ is compact. Finally,
$S$ being continuous and $G'(u)$ compact, then $S\circ G'(u)$ is compact and the proof is complete.
\end{proof}
The third property proves that the application $F'(u) : X\rightarrow X$ is an isomorphism, where $F$ is defined by (\ref{F'}) and $u\in X$.
\blem \label{lcom1} Let $u\in X$.
Assuming that $G$ is defined by (\ref{g}) and  $S$ defined by $(\ref{s})$, then 
the map $F'(u) : X\to X$ is an isomorphism of $X$, where $F$ is defined by $(\ref{F'})$ and $X=H_0^2(\Omega)$.
\elem
\begin{proof} 
By assumption, we have $F(u)=u-S\circ G(u)$ and 
the function $F$ is differentiable as the sum of a differentiable function and the composite of two
differentiable functions. Since $S\in \mathcal{L}(Y,X)$, we have $F'(u)=I-S'[G(u)]\circ G'(u)$, that is 
$F'(u)=I-S\circ G'(u)$, which implies $F'(u)\cdot v=v-S\circ G'(u)\cdot v$.
According to Lemma \ref{lcom}\  $S\circ G'(u)$ is a compact operator, and 
from the Fredholm Alternative \cite{Brezis}, $F'(u)$ is an isomorphism on $V$ if the equation  $[I-S\circ G'(u)](w)=0$ admits a unique solution $w=0$. The condition $[I-S\circ G'(u)](w)=0$ equivalent to 
$[S\circ G'(u)] (w)=w$.
By definition of $S$,
$S[G'(u)\cdot w]=w$ implies that $w$ is the solution to the problem:
\begin{equation}
\begin{cases}
\Delta^2 w= G'(u)\cdot w\text{ in } \Omega\\
w=0=\dfrac{\partial w}{\partial n} \text{ on } \Gamma.
\end{cases}
\end{equation}
 It is well known that $G'(u)\cdot w=-\lambda(2p+1)|u|^{2p}w$ and so, we obtain auxiliary system,
\begin{equation}
\begin{cases}
\Delta^2 w=-\lambda(2p+1)|u|^{2p}w \text{ in } \Omega\\
w=\dfrac{\partial w}{\partial n}=0\text{ on } \Gamma,
\end{cases}
\end{equation}
which implies,
\begin{equation}
\label{po}
\begin{cases}
\Delta^2 w+\lambda(2p+1)|u|^{2p}w=0 \text{ in } \Omega\\
w=\dfrac{\partial w}{\partial n}=0 \text{ on } \Gamma,
\end{cases}
\text{ with } \lambda>0.
\end{equation}
According to Proposition \ref{pro1}, we deduce that 
$[I-S\circ G'(u)](w)=0$ is equivalent to $(\ref{po})$.
Consequently,  $[I-S\circ G'(u)](w)=0$ admits a unique solution $w=0$ if and only if $(\ref{po})$ admits a unique solution $w=0$. Let us then show that the problem $(\ref{po})$ admits
a unique solution $w=0$. The problem $(\ref{po})$ is equivalent to $a(v,w)=0, \forall v\in H_0^2(\Omega):=X$, 
with 
\begin{equation}
a(v,w):=\displaystyle\int_{\Omega}\Delta w\Delta vdx+\lambda(2p+1)|u|^{2p}\int_{\Omega}wvdx, \forall (v,w)\in X^2.
\end{equation}
 Since $a\in \mathcal{L}(X\times X,\mathbb{R})$ and $a(w,w)\geq
\left[\min(1,\lambda(2p+1)\|u\|_{L^{2p}(\Omega)}^{2p})\right]|w|_{2,\Omega}^2, \forall w\in X
 $,
 we deduce, according to Lax-Milgram lemma, $(\ref{po})$ admits a unique solution \\$w\in X$.
 Since $w=0$ is a solution  of PDE $(\ref{po})_1$ and moreover $w=\dfrac{\partial w}{\partial n}=0 \text{ on } \Gamma$ then $(\ref{po})$ admits a unique solution $w=0$. This completes the proof.
\end{proof}
Finally, the last result is given by the Lemma \ref{lla} below:
\blem \label{lla} Let $S\in \mathcal{L}(V',V)$ define by (\ref{s}) and 
 $S_h\in \mathcal{L}(V',V_h)$ such that for all $f\in V',\ S_h(f)=w_h$  where $w_h$ the solution of the discrete
problem: find $w_h\in V_h \text{ such that }$
$$\displaystyle\int_{\Omega}\Delta w_h\Delta v_hdx=\langle f,v_h\rangle \ \forall v_h\in V_h.$$ Then,
$\lim\limits_{h\to 0}\|S_h-S\|_{\mathcal{L}(V',V)}=0$, where $V=H_0^2(\Omega)$ and $V'=H^{-2}(\Omega)$.
\elem
\begin{proof}
Let's remember that $$\|S_h-S\|_{\mathcal{L}(V',V)}=\sup_{f\in V'^*}\dfrac{\|S_h(f)-S(f)\|_V}{\|f\|_{V'}}.$$
We have, $\|S_h(f)-S(f)\|_V=\|w_h-w\|_V$ and 
as $w_h$ is the approximate solution of $w$, we deduce
$\lim\limits_{h\to 0}\|w_h-w\|_{V}=0$, which implies
$\lim\limits_{h\to 0}\|S_h(f)-S(f)\|_{V}=0.$
Hence,
$$\lim\limits_{h\to 0}\|S_h-S\|_{\mathcal{L}(V',V)}=0.$$
\end{proof}

From all of the above, we draw the following consequence which is nothing other than the conclusion of Theorem \ref{th}.
\bcoro Let $u$ be the solution of problem $(\ref{p})$ and $\Pi_h\in \mathcal{L}(V,V_h)$ a linear operator
verifying  $\displaystyle\lim\limits_{h\to 0}|v-\Pi_hv|_{2,\Omega}=0$.
There then exists an original neighborhood $\mathcal{O}$ in $V$ and a real number $h_0>0$ such that for
all $h\leqslant h_0$, the discrete problem $(\ref{pd})$ admits a unique solution $u_h$ with $u_h-u\in \mathcal{O}$. Additionally, we have the following a priori error estimate : 
\begin{equation}
\label{}
\|u_h-u\|_V\leqslant M\left(\|u-\Pi_h u\|_V + \|(S_h-S)\circ G(u)\|_V\right).
\end{equation}
with $M$ a constant independent of $h$.
\ecoro
\subsection{BELL and HSIEH-CLOUGH-TOCHER finite elements}\label{Bell}
The conforming approximation of fourth order problems needs finite elements of class $\mathcal{C}^1$. Focusing our attention to triangular finite elements and, in particular, to those which use polynomial spaces, we use in this work two families \cite{BH}: BELL triangles and HSIEH-CLOUGH-TOCHER triangles.\\
Let us consider $(\mathcal{T}_h )_{h>0}$ a family of conforming isotropic triangulation of $\bar\Omega$, where $\Omega$ an open bounded polygonal boundary of $\R^2$. Namely, we set 
$\overline{\Omega}=\displaystyle\bigcup_{K\in\cT_h}K$ where $K$ is triangle, and 
$h_K/\rho_K\leq \sigma_0$ for all element $K$, for all $h$. The quantities $h_K$ and $\rho_K$ are diameter of $K$ and diameter of the biggest ball contained in $K$ respectively (see Figs. \ref{isotropic}, \ref{adm1} and \ref{adm2} for illustration).
\begin{figure}[http!]
	\begin{minipage}[c]{.30\textwidth}
		\centering
		\begin{center}
			\begin{tikzpicture}[scale=0.5]
			\draw (0,1)--(7,1);
			\draw (0,1)--(2,4.5);
			\draw (7,1)--(2,4.5);
			\draw [line width=0.75pt] [>=latex,<->](0,0.75)--(7,0.75)node [midway,below,sloped] {$\mbox{diam} (K)=h_K$};
			\draw  (2.45,2.43) circle (1.4);
			\draw (2.45,2.) node [above]{$\bullet$};
			\draw (3,1.5) node [above]{\small{$\rho_K$}};
			\draw  [line width=0.75pt](2.45,2.43)--(2.5,1);
			\end{tikzpicture}
		\end{center}
		\caption{\footnotesize{\small\small{Isotropic element $K$ in $\mathbb{R}^2$.}}}
		\label{isotropic}
	\end{minipage}
	\hspace*{0.3cm}
	\begin{minipage}[c]{.30\textwidth}
		\centering
		\begin{center}
			\begin{tikzpicture}[scale=0.5]
			\draw (0,3)--(0,-3);
			\draw (0,3)--(1,0);
				\draw (1,0)--(0,-3);
				\draw (0,3)--(-1,0);
					\draw (-1,0)--(0,-3);
			\draw (0,3)--(2,0);
			\draw (0,3)--(-2,0);
			\draw (-2,0)--(2,0);
			\draw (2,0)--(0,-3);
			\draw (-2,0)--(0,-3);
			\draw (0,3)node {$\bullet$};
			\draw (0,-3)node {$\bullet$};
			\draw (2,0)node {$\bullet$};
			\draw (-2,0)node {$\bullet$};
			\draw (1,0)node {$\bullet$};
			\draw (-1,0)node {$\bullet$};
			\draw (0,0)node {$\bullet$};
			\end{tikzpicture}
		\end{center}
		\caption{\footnotesize{\small\small Example of conforming mesh in $\mathbb{R}^2$}}
		\label{adm1}
	\end{minipage}
	\hspace*{0.3cm}
	\begin{minipage}[c]{.30\textwidth}
		\centering
		\begin{center}
			\begin{tikzpicture}[scale=0.5]
			\draw (0,3)--(1,0);
			\draw (0,3)--(-1,0);
			\draw (0,3)--(2,0);
			\draw (0,3)--(-2,0);
			\draw (-2,0)--(2,0);
			\draw (2,0)--(0,-3);
			\draw (-2,0)--(0,-3);
			\draw (1,0)node {$\bullet$};
			\draw (-1,0)node {$\bullet$};
			\draw (0,-3)-- (0,0);
			\draw (0,0)node {$\bullet$};
			\end{tikzpicture}
		\end{center}
		\caption{\footnotesize{\small\small Example of nonconforming mesh in $\mathbb{R}^2$ }}
		\label{adm2}
	\end{minipage}
\end{figure}\mbox{ }
 As in the standard theory, a finite element (in Ciarlet sens) is denoted by a triplet $(K,P_K,\Sigma_K)$ where $K$ is a compact domain of $\mathbb{R}^2$ with $\stackrel{o}{K}$ not empty and $\partial K:=\overline{K}\smallsetminus \stackrel{o}{K}$. The set $P_K$ denotes a space of functions, and $\Sigma_K$ is a set of functional of $P_K^{*}$ (space of linear form defined on $P_K$) \cite[Section 6.1]{Nicaise}.
 We 
define the approximation space $V_h$ as follows:
\begin{equation}
\label{vh}
V_h=\{v\in C^1(\bar{\Omega})\text{ such that } v_{|_{K}}\in P_K,\ \forall K\in\mathcal{T}_h\}\cap H_0^2(\Omega).
\end{equation}
\subsubsection{BELL finite elements}
The ARGYRIS triangle is used to complete polynomial of degree five as function space. By suppression of the values of the normal slopes at the three midside nodes, one gets the BELL triangle (see Table \ref{BelT}). The corresponding basis functions of these elements were done by ARGYRIS-FRIED-SCHARPF elements \cite{AR1, AR2} and next slightly corrected in ARGYRIS-SCHARPF \cite{AR3}. The authors in \cite{AR1, AR2, AR3} achieved considerable simplifications by using the so-called eccentricity parameters which permit to take into account the normal derivatives at the midside nodes (explicitly for ARGYRIS triangle and implicitly for BELL triangle) for triangles of any shape.
\vspace{5mm}
\subsubsection{HSIEH-CLOUGH-TOCHER finite elements}
The HCT finite elements complete and reduced use piecewise polynomials of third degree \cite{BH,Clough}. These elements give rise to interpolations of Hermite type and they permit the construction of spaces of approximations functions of $\mathcal{C}^1$ class. The combined employ of barycentric coordinates $\lambda_{i}$ and eccentricity parameters $E_{i}$ enables the finite element to be defined for any triangle not involving the notion of a reference finite element. Their characteristics are that the triangle is subdivided in three subtriangle  using (for exemple) the center of gravity, and on each subtriangle, we use polynomials of degree three so that the resulting function is of class $\mathcal{C}^1$ on the assembled triangle (see Tables \ref{complet} and \ref{Reduced}).  We present in this paper, a set of basis functions for both elements, \emph{complete} or \emph{reduced}, for triangles of any shape and we use the eccentricity parameters to define the normal slope at midside nodes. These parameters are only dependent on the coordinates of the vertices of the triangle. For simplicity, we shall denote these elements, HCT-C triangle for \emph{complete}  element and HCT-R triangle for \emph{reduced} element, respectively. We denote 
$b_i$ the respective midpoints of sides  $a_{i+1}a_{i}$  and  
$E_i=\dfrac{l^2_{i+2}-l^2_{i+1}}{l^2_{i}}; i=1,2,3,$	
where $a_i=(x_i,y_i)$,
$l^2_{i}=(x_{i+2}-x_{i+1})^2+(y_{i+2}-y_{i+1})^2; i=1,2,3 $.
\begin{table}[H]
	\label{T1}
	\begin{tabular}{|p{1cm}|c|l|}
		\hline
		Reference elements & Space $P_{\hat{K}}$; Degrees of freedom&Basic functions\\
		\hline
		\multicolumn{1}{|c|}{\begin{tikzpicture}
			\draw (0,0) circle (0.2);
			\draw (0,0) circle (0.3);
			\draw (3,0) circle (0.2);
			\draw (3,0) circle (0.3);
			\draw (0,3) circle (0.2);
			\draw (0,3) circle (0.3);
			\coordinate (A) at (0,0);
			\coordinate (B) at (3,0);
			\coordinate (C) at (0,3); 
			
			
			\draw[thick] (A) -- (B) -- (C) -- cycle;
			
			
			\filldraw[black] (A) circle (3pt);
			\filldraw[black] (B) circle (3pt);   
			\filldraw[black] (C) circle (3pt);
			\filldraw[black] (1,1.3)  node[below]{$\hat{K}$};
			\filldraw[black] (0,-0.4)  node[below]{$\hat{a_1}$};
			\filldraw[black] (3.4,-0.4)  node[below]{$\hat{a_2}$};   
			\filldraw[black] (0,3.3)  node[above]{$\hat{a_3}$};

			\end{tikzpicture}
		} & $\begin{array}{ccc}
		
		P_{\hat{K}}=\{\hat{p}\in C^1(\hat{K})|\hat{p}_{|\hat{K}}\in \mathbb{P}_5(\hat{K})\}\\\\ \Sigma_{\hat{K}}=\big\{\hat{p}(\hat{a_i}),\dfrac{\partial \hat{p}}{\partial \hat{x}}(\hat{a_i}),\dfrac{\partial \hat{p}}{\partial \hat{y}}(\hat{a_i}),\\\\\dfrac{\partial^2 \hat{p}}{\partial \hat{x}^2}(\hat{a_i}),\dfrac{\partial^2 \hat{p}}{\partial \hat{x} \partial \hat{y}}(\hat{a_i}),\dfrac{\partial^2 \hat{p}}{\partial \hat{y}^2}(\hat{a_i})\big\}\\\\ \dim P_{\hat{K}}=18\end{array}$&\small$\begin{array}{lll}
		\mbox{}\\\\
		\phi_{1,1}=\lambda^2(10\lambda-15\lambda^2+6\lambda^3\\+30\hat{x}\hat{y}(\hat{x}+\hat{y}))\\\\
		\phi_{1,2}=\hat{x}\lambda^2(3-2\lambda-3\hat{x}^2+6\hat{x}\hat{y})\\\\
		\phi_{1,3}=\hat{y}\lambda^2(3-2\lambda-3\hat{y}^2+6\hat{x}\hat{y})\\\\
		\phi_{1,4}=\dfrac{1}{2}\lambda^2\hat{x}^2(1-\hat{x}+2\hat{y})\\\\
		\phi_{1,5}=\hat{x}\hat{y}\lambda^2\\\\
		\phi_{1,6}=\dfrac{1}{2}\lambda^2\hat{y}^2(1-\hat{y}+2\hat{x})\\\\\
		\phi_{2,1}=\hat{x}^2(10\hat{x}-15\hat{x}^2+\hat{x}^3\\+15\hat{y}^2\lambda)\\\\
		\phi_{2,2}=\dfrac{1}{2}\hat{x}^2\hat{y}(6-4\hat{x}-3\hat{y}-3\hat{y}^2\\+3\hat{y}\hat{x})\\\
		\phi_{2,3}=\dfrac{1}{2}\hat{x}^2(-8\hat{x}+14\hat{x}^2-6\hat{x}^3\\-15\hat{y}^2\lambda)\\\\
		\phi_{2,4}=\dfrac{1}{2}\hat{x}^2(2\hat{x}(1-\hat{x})^2+5\hat{y}^2\lambda)\\\\
		\phi_{2,4}=\dfrac{1}{2}\hat{x}^2\hat{y}(-2+2\hat{x}+\hat{y}+\hat{y}^2\\-\hat{y}\hat{x})\\\\
		\phi_{2,6}=\dfrac{1}{4}\hat{x}^2\hat{y}^2\lambda+\dfrac{1}{2}\hat{x}^3\hat{y}^2\\\\
		\phi_{3,1}=\hat{y}^2(10\hat{y}-15\hat{y}^2+6\hat{y}^3\\+15\hat{x}^2\lambda)\\\\
		\phi_{3,2}=\dfrac{1}{2}\hat{x}\hat{y}^2(-6-3\hat{x}-4\hat{y}-3\hat{x}^2\\+3\hat{y}\hat{x})\\\\
		\phi_{3,3}=\dfrac{1}{2}\hat{y}^2(-8\hat{y}+14\hat{y}^2-6\hat{y}^3\\-15\hat{x}^2\lambda)\\\\
		\phi_{3,4}=\dfrac{1}{4}\hat{x}^2\hat{y}^2\lambda+\dfrac{1}{2}\hat{x}^2\hat{y}^3\\\\
		\phi_{3,5}=\dfrac{1}{2}\hat{x}\hat{y}^2(-2+\hat{x}+2\hat{y}\\+\hat{x}^2-\hat{y}\hat{x})\\\\
		\phi_{3,6}=\dfrac{1}{4}\hat{y}^2(2\hat{y}(1-\hat{y})^2+5\hat{x}^2\lambda)
	\end{array}$\\
	\hline
	\end{tabular}
	\vspace*{0.5cm}
	\caption{Finite element of \textbf{Bell} \cite{ AR3} }
	\label{BelT}
	\end{table}
\begin{table}[H]
	\label{T2}	
	\begin{tabular}{|p{2cm}|l|l|}
		\hline
		\small Reference elements & \small Space $P_{K}$; Degrees of freedom&\small Basic functions\\
		\hline
		\multicolumn{1}{|c|}{\small\small\begin{tikzpicture}
			\coordinate (A1) at (0,0);
			\coordinate (A2) at (5,0);
			\coordinate (A3) at (2,4);
			
			\coordinate (B) at (2.0,1.8);
			
			\draw (2.2,0) node[below] {$b_3$};
			\draw (0.9,1.8) node[left] {$b_2$};
			\draw (3.3,2.4) node[right] {$b_1$};
			\draw (2.4,0) node {$|$};
			\draw (1,1.8) node {$-$};
			\draw (3.2,2.2) node {$-$};
			
			\filldraw[black] (2.0,0.4) node[above] {$K_1$};
			\filldraw[black] (1.2,2.4) node[right] {$K_2$};
			\filldraw[black] (2.7,2.5) node[below] {$K_3$};
			
			\coordinate (M12) at (2.6,0);
			\coordinate (M23) at (3.5,2);
			\coordinate (M31) at (1,2);
			
			\draw (M12) rectangle +(0.2,0.2);
			\draw (1.0,2) rectangle +(0.2,0.2);
			\draw (3.4,1.77) rectangle +(0.2,0.2);
			
			\draw[thick] (A1) -- (A2) -- (A3) -- cycle;
			
			\draw[thick] (A1) -- (B);
			\draw[thick] (A2) -- (B);
			\draw[thick] (A3) -- (B);
			
			\draw[dashed,thick] (M12) -- (A3);
			\draw[dashed,thick] (M23) -- (A1);
			\draw[dashed,thick] (M31) -- (A2);
			
			\fill[black] (A1) circle (2pt) node[below] {$a_1$};
			\fill[black] (A2) circle (2pt) node[below] {$a_2$};
			\fill[black] (A3) circle (2pt) node[above] {$a_3$};
			
			\fill[black] (B) circle (2pt) node[below] {$a_0$};
			
			\fill[black!100] (M12) circle (2pt) node[below] {$c_3$};
			\fill[black!100] (M23) circle (2pt) node[right] {$c_1$};
			\fill[black!100] (M31) circle (2pt) node[left] {$c_2$};
			\end{tikzpicture}}
			&\small $\begin{array}{ccc}
		P_K=\{p\in C^1(K);p_{|{K_i}}\in \mathbb{P}_3(K_i),\\ i=1,2,3\}\\\\
		\Sigma_{K}=\{p(a_i),\nabla p(a_i)\cdot(a_{i+1}-a_i),\\\nabla p(a_i)\cdot(a_{i+2}-a_i),\\\nabla p(b_i)\cdot(a_i-c_i), i=1,2,3 \}\\\\
		\dim P_K=12\\
		a_0= \text{ barycenter of }K\\
		\text{ Let } r_i=p_{|{K_i}}
		\end{array}$&\small$\begin{array}{lllll}
		r^0_{i,i}=-\frac{1}{2}(E_{i+1}-E_{i+2})\lambda^3_{i}\\+\frac{3}{2}(3+E_{i+1})\lambda^2_{i}\lambda_{i+2}\\+\frac{3}{2}(3-E_{i+2})\lambda^2_{i}\lambda_{i+1}\\\\
		r^0_{i,i+1}=\frac{1}{2}(1-2E_{i}-E_{i+1})\lambda^3_{i}\\+\lambda^3_{i+1}-\frac{3}{2}(1-E_{i})\lambda^2_{i}\lambda_{i+2}\\+\frac{3}{2}(E_{i}+E_{i+2})\lambda^2_{i}\lambda_{i+1}\\+3\lambda^2_{i+1}\lambda_{i}+3\lambda^2_{i+1}\lambda_{i+2}\\+3(1-E_{i})\lambda_{i}\lambda_{i+1}\lambda_{i+2}\\\\
		r^0_{i,i+2}=\frac{1}{2}(1+2E_{i}+E_{i+1})\lambda^3_{i}\\+\lambda^3_{i+2}-\frac{3}{2}(E_{i}+E_{i+1})\lambda^2_{i}\lambda_{i+2}\\-\frac{3}{2}(1+E_{i})\lambda^2_{i}\lambda_{i+1}\\+3\lambda^2_{i+2}\lambda_{i+1}+3\lambda^2_{i+2}\lambda_{i}\\+3(1+E_{i})\lambda_{i}\lambda_{i+1}\lambda_{i+2}\\\\
		r^1_{i,i,i+2}=-\frac{1}{12}(1+E_{i+1})\lambda^3_{i}\\+\frac{1}{4}(7+E_{i+1})\lambda^2_{i}\lambda_{i+2}\\-\frac{1}{2}\lambda^2_{i}\lambda_{i+1}\\\\
		r^1_{i,i,i+1}=-\frac{1}{12}(1-E_{i+2})\lambda^3_{i}\\-\frac{1}{2}\lambda^2_{i}\lambda_{i+2}+\frac{1}{4}(7-E_{i+2})\lambda^2_{i}\lambda_{i+1}\\\\\\
		r^1_{i,i+1,i}=-\frac{1}{12}(7+E_{i+2})\lambda^3_{i}\\+\frac{1}{2}\lambda^2_{i}\lambda_{i+2}+\frac{1}{4}(5+E_{i+2})\lambda^2_{i}\lambda_{i+1}\\+\lambda^2_{i+1}\lambda_{i}-\lambda_{i}\lambda_{i+1}\lambda_{i+2}\\\\
		r^1_{i,i+1,i+2}=\frac{1}{6}(4-E_{i})\lambda^3_{i}\\-\frac{1}{4}(3-E_{i})\lambda^2_{i}\lambda_{i+1}\\-\frac{1}{4}(5-E_{i})\lambda^2_{i+1}\lambda_{i}+\lambda^2_{i+2}\lambda_{i+1}\\+\frac{1}{2}(3-E_{i})\lambda_{i}\lambda_{i+1}\lambda_{i+2}\\\\
		r^1_{i,i+2,i+1}=\frac{1}{6}(4+E_{i})\lambda^3_{i}\\-\frac{1}{4}(5+E_{i})\lambda^2_{i}\lambda_{i+2}\\-\frac{1}{4}(3+E_{i})\lambda^2_{i}\lambda_{i+1}+\lambda^2_{i+2}\lambda_{i+1}\\+\frac{1}{2}(3+E_{i})\lambda_{i}\lambda_{i+1}\lambda_{i+2}\\\\
		r^1_{i,i+2,i}=-\frac{1}{12}(7-E_{i+1})\lambda^3_{i}\\+\frac{1}{4}(5-E_{i+1})\lambda^2_{i}\lambda_{i+2}+\frac{1}{2}\lambda^2_{i}\lambda_{i+1}\\+\lambda^2_{i+2}\lambda_{i}-\lambda_{i}\lambda_{i+1}\lambda_{i+2}\\\\
		r^1_{\bot,i,i}=\frac{4}{3}\lambda^3_{i}-2\lambda^2_{i}\lambda_{i+2}-2\lambda^2_{i}\lambda_{i+1}\\+4\lambda_{i}\lambda_{i+1}\lambda_{i+2}\\\\
		r^1_{\bot,i,i+1}=-\frac{2}{3}\lambda^3_{i}+2\lambda^2_{i}\lambda_{i+2}\\\\
		r^1_{\bot,i,i+2}=-\frac{2}{3}\lambda^3_{i}+2\lambda^2_{i}\lambda_{i+1}\\\\
		\end{array}$\\
		\hline	 
	\end{tabular}
\vspace*{0.5cm}
	\caption{Finite element of \textbf{Clough-Tocher Complete} \cite{BH} }
	\label{complet}
\end{table}
\newpage
\begin{table}[H]
	\label{T3}	
	\begin{tabular}{|p{2cm}|c|r|}
		\hline
		Reference elements & Space $P_{K}$; Degrees of freedom&Basic functions\\
		\hline
		\multicolumn{1}{|c|}{\begin{tikzpicture}

			\coordinate (A) at (0,0);
			\coordinate (B) at (4,0);
			\coordinate (C) at (2,3.464);
			
			\coordinate (G) at (2.2, 1);

			\draw[thick] (A) -- (B) -- (C) -- cycle;
			
			\draw[thick] (A) -- (G) -- (B);
			\draw[thick] (B) -- (G) -- (C);
			\draw[thick] (C) -- (G) -- (A);

			\filldraw[black] (A) circle (2pt) node[below]{$a_1$};
			\filldraw[black] (B) circle (2pt) node[below]{$a_2$};   
			\filldraw[black] (C) circle (2pt) node[above]{$a_3$};
			\filldraw[black] (G) circle (2pt) node[below]{$a_0$};	
			\end{tikzpicture}} &\small $\begin{array}{ccc} P_K=\{p\in C^1(K);\\r_i=p_{|{K_i}}\in \mathbb{P}_3(K_i),\\ i=1,2,3 \text{ et }\dfrac{\partial p}{\partial n_i}\in P_1(K'_i),\\ \forall K'\subset K\}\\\\ \Sigma_{K}=\{p(a_i),\nabla p(a_i)\cdot(a_{i+1}-a_i),\\\nabla p(a_i)\cdot(a_{i+2}-a_i),\\ i=1,2,3 \}\\\\
		\dim P_K=9\\
		a_0= \text{ barycenter of }K
		\end{array}$&\small$\begin{array}{lllll}   
		\check{r}^0_{i,i}=-\frac{1}{2}(E_{i+1}-E_{i+2})\lambda^3_{i}\\+\frac{3}{2}(3+E_{i+1})\lambda^2_{i}\lambda_{i+2}\\+\frac{3}{2}(3-E_{i+2})\lambda^2_{i}\lambda_{i+1}\\\\
		\check{r}^0_{i,i+1}=\frac{1}{2}(1-2E_{i}-E_{i+1})\lambda^3_{i}\\+\lambda^3_{i+1}-\frac{3}{2}(1-E_{i})\lambda^2_{i}\lambda_{i+2}\\+\frac{3}{2}(E_{i}+E_{i+2})\lambda^2_{i}\lambda_{i+1}\\+3\lambda^2_{i+1}\lambda_{i}+3\lambda^2_{i+1}\lambda_{i+2}\\+3(1-E_{i})\lambda_{i}\lambda_{i+1}\lambda_{i+2}\\\\
		\check{r}^0_{i,i+2}=\frac{1}{2}(1+2E_{i}+E_{i+1})\lambda^3_{i}\\+\lambda^3_{i+2}-\frac{3}{2}(E_{i}+E_{i+1})\lambda^2_{i}\lambda_{i+2}\\-\frac{3}{2}(1+E_{i})\lambda^2_{i}\lambda_{i+1}\\+3\lambda^2_{i+2}\lambda_{i+1}+3\lambda^2_{i+2}\lambda_{i}\\+3(1+E_{i})\lambda_{i}\lambda_{i+1}\lambda_{i+2}\\\\
		\check{r}^1_{i,i,i+2}=-\frac{1}{4}(1+E_{i+1})\lambda^3_{i}\\+\frac{1}{4}(5+3E_{i+1})\lambda^2_{i}\lambda_{i+2}\\+\frac{1}{2}\lambda^2_{i}\lambda_{i+1}\\\\
		\check{r}^1_{i,i,i+1}=-\frac{1}{4}(1-E_{i+2})\lambda^3_{i}\\+\frac{1}{2}\lambda^2_{i}\lambda_{i+2}+\frac{1}{4}(5-3E_{i+2})\lambda^2_{i}\lambda_{i+1}\\\\
		\check{r}^1_{i,i+1,i}=\frac{1}{4}(1-E_{i+2})\lambda^3_{i}\\-\frac{1}{2}\lambda^2_{i}\lambda_{i+2}-\frac{1}{4}(1-3E_{i+2})\lambda^2_{i}\lambda_{i+1}\\+\lambda^2_{i+1}\lambda_{i}+\lambda_{i}\lambda_{i+1}\lambda_{i+2}\\\\
		\check{r}^1_{i,i+1,i+2}=-\frac{1}{2}E_{i}\lambda^3_{i}\\-\frac{1}{4}(1-3E_{i})\lambda^2_{i}\lambda_{i+2}\\+\frac{1}{4}(1+3E_{i})\lambda^2_{i}\lambda_{i+1}+\lambda^2_{i+1}\lambda_{i+2}\\+\frac{1}{2}(1-3E_{i})\lambda_{i}\lambda_{i+1}\lambda_{i+2}\\\\
		\check{r}^1_{i,i+2,i+1}=\frac{1}{2}E_{i}\lambda^3_{i}\\+\frac{1}{4}(1-3E_{i})\lambda^2_{i}\lambda_{i+2}\\-\frac{1}{4}(1+3E_{i})\lambda^2_{i}\lambda_{i+1}+\lambda^2_{i+2}\lambda_{i+1}\\+\frac{1}{2}(1+3E_{i})\lambda_{i}\lambda_{i+1}\lambda_{i+2}\\\\
		\check{r}^1_{i,i+2,i}=\frac{1}{4}(1+E_{i+1})\lambda^3_{i}\\-\frac{1}{4}(1+3E_{i+1})\lambda^2_{i}\lambda_{i+2}-\frac{1}{2}\lambda^2_{i}\lambda_{i+1}\\+\lambda^2_{i+2}\lambda_{i}+\lambda_{i}\lambda_{i+1}\lambda_{i+2}\\\\
		\end{array}$\\
		\hline	 
	\end{tabular}
\vspace*{0.5cm}
	\caption{\textbf{Finite element of Tocher Reduced} \cite{Clough} }
	\label{Reduced}
\end{table}

\subsection{Iterative problem}\label{algorithme}
 We recall that $G$ is defined in (\ref{g}) and $S$  define by (\ref{s}). We consider $F$, application define by (\ref{F'}).
Our continuous problem being nonlinear, we use an algorithm to solve the approximate problem. In the previous sub-section \ref{technique}, we showed that the problem $(\ref{p})$ is equivalent to $(\ref{F})$ (cf. Poposition \ref{pro1}) and that the approximate problem $(\ref{Fh})$ admits a unique solution thanks to the Brezzi-Rappaz-Raviart theorem (Theorem \ref{th}). Thus the equation $u_h-S_h\circ G(u_h)=0$ admits a
unique solution in a neighborhood $\mathcal{O}$ of the origin in $V$.
\begin{equation}\label{fi}
F(u_h)=0\mbox{ equivalent to } S_h\circ G(u_h)=u_h.\end{equation}
Setting $T(u_h)=S_h\circ G(u_h)$, 
we thus have $F(u_h)=0$ equivalent to  $T(u_h)=u_h$.
We deduce that $u_h$ is the unique fixed point of $T$.
From $T(u_h)=u_h$, we define the recurring sequence for a $u^0_h$ fixed and for all $n\geqslant0$, $u^{n+1}_h=T(u^n_h)$ with $T=S_h\circ G$ and $(u^n_h)_{n\geqslant 0}$ a sequence of solutions of $(\ref{fi})$. Thus, we use the fixed point algorithm to calculate $u_h$.
\subsection{Fixed point algorithm}\label{pointfixe}
Let $u^{0}_h$ initially known. For $n\in \N$, the fixed point algorithm  is presented as follows: given $u_h^n\in V_h$, find  $u^{n+1}_h\in V_h \text{ such that }$
\begin{equation}\label{al}
\int_{\Omega}\Delta u^{n+1}_h\Delta v_hdx+\lambda\int_{\Omega}|u^{n}_h|^{2p}u^{n+1}_h v_hdx=\langle f,v_h\rangle \ \forall v_h\in V_h.
\end{equation}
Finally, the Lax-Milgram lemma applied to the problem (\ref{al}) ensure unicity of the solution $u_h^{n+1}$ in $V_h$. 
In addition we have the following estimate : 
\begin{equation}
|u^{n+1}_h|_{2,\Omega} \leqslant\|f\|_{-2,\Omega}, \forall n\in\N.
\end{equation}
After showing the existence and uniqueness of the sequence
of functions $(u_h^n)_{n\in \N}$, we show
that it converges to $u_h$ by establishing the estimate,
\begin{equation}
|u^{n+1}_h-u_h|_{2,\Omega}\leqslant C|u^{n}_h-u_h|_{2,\Omega}.
\end{equation}
\bthm\label{thm1}
 Let $u^{n+1}_h$ the solution of the iterative problem $(\ref{al})$ and $u_h$ the solution of the discrete problem $(\ref{pd})$. Let $p$ be a positive real number. If $(1-\lambda A)>0$,  the problem $(\ref{al})$ verifies the following estimate:
$$|u^{n+1}_h-u_h|_{2,\Omega}\leqslant \lambda(1-\lambda A)^{-1}B|u^{n}_h-u_h|_{2,\Omega},$$
with $A=mes(\Omega)c^{2p}_1c_2\|f\|^{2p}_{-2,\Omega}$ and $B= mes(\Omega)4pC^{2p}_1c'_2c_3\|f\|^{2p}_{-2,\Omega}$.
\ethm
\begin{proof} (Theorem \ref{thm1})
From the problem $(\ref{al})$ we have
$$\int_{\Omega}\Delta u^{n+1}_h\Delta v_hdx+\lambda\int_{\Omega}|u^{n}_h|^{2p}u^{n+1}_hv_hdx =\langle f,v_h\rangle \ \forall v_h\in V_h,$$
and 
from the problem $(\ref{pd})$ we have
$$\int_{\Omega}\Delta u_h\Delta v_hdx+\lambda\int_{\Omega}|u_h|^{2p}u_hv_hdx =\langle f,v_h\rangle \ \forall v_h\in V_h.$$
By differentiating between the equality of problems $(\ref{al})$ and $(\ref{pd})$ we obtain
$$\int_{\Omega}\Delta( u^{n+1}_h-u_h)\Delta v_hdx=-\lambda\int_{\Omega}(|u^{n}_h|^{2p}u^{n+1}_h-|u_h|^{2p}u_h)v_hdx.$$
By adding and subtracting $|u^n_h|^{2p}u_h$ we have
\begin{eqnarray}\label{E}\nonumber
\int_{\Omega}\Delta( u^{n+1}_h-u_h)\Delta v_hdx&=&-\lambda\int_{\Omega}(|u^{n}_h|^{2p}u^{n+1}_h-|u^n_h|^{2p}u_h)v_hdx\\\nonumber
&&-\lambda\int_{\Omega}(|u^n_h|^{2p}u_h-|u_h|^{2p}u_h)v_hdx\\
&=&\lambda\int_{\Omega}|u^n_h|^{2p}(u_h-u^{n+1}_h)v_hdx\\\nonumber
&&+\lambda\int_{\Omega}(|u_h|^{2p}-|u^n_h|^{2p})u_hv_hdx.	
\end{eqnarray}
Let's mark each term. Using the fact that $H^2(\Omega)$ injects into $C^0(\bar{\Omega})$ we have:
$|u^n_h|\leqslant c_1|u^n_h|_{2,\Omega} \ \forall u\in H^{2}_{0}(\Omega),$ and moreover $|u^{n+1}_h|_{2,\Omega} \leqslant\|f\|_{-2,\Omega}\ \forall n\in\N$, which implies the estimation:
\begin{eqnarray}\label{est1}
\lambda\int_{\Omega}|u^n_h|^{2p}(u_h-u^{n+1}_h)v_hdx \leq A|v_h|_{1,\Omega}|u^{n+1}_h-u_h|_{1,\Omega}, 
\end{eqnarray}
where
$A=\mbox{mes}(\Omega)c^{2p}_1c_2c_3\|f\|^{2p}_{-2,\Omega}> 0.$
For estimate the term
$\lambda\int_{\Omega}(|u_h|^{2p}-|u^n_h|^{2p})u_hv_hdx$, we 
use the fact that  $||u_h|^{2p}-|u^n_h|^{2p}|\leqslant 2p|u_h-u^n_h|(|u_h|^{2p-1}+|u^n_h|^{2p-1})$, and we obtain the estimate:
\begin{eqnarray}\label{est2}
	\lambda\int_{\Omega}(|u_h|^{2p}-|u^n_h|^{2p})u_hv_hdx \leq \lambda B |v_h|_{2,\Omega}|u_h-u^n_h|_{2,\Omega}, 
\end{eqnarray}
where $B=4p\mbox{mes}(\Omega)c^{2p}_1c'_2c_3\|f\|^{2p}_{-2,\Omega}.$
Combining $(\ref{est1})$ and $(\ref{est2})$, and 
replacing in the inequality $(\ref{E})$ $v_h$ by $u^{n+1}_h-u_h$ we obtain:
\begin{eqnarray}\label{E1}\nonumber
\int_{\Omega}\Delta^2( u^{n+1}_h-u_h)dx&\leqslant& \lambda A|u^{n+1}_h-u_h|^2_{2,\Omega}+\lambda B|u^{n+1}_h-u_h|_{2,\Omega}|u_h-u^n_h|_{2,\Omega},
\end{eqnarray}
Using the equivalence between the norms $\|\Delta u\|_{L^2(\mathcal{O})}$ and $|\cdot|_{2,\Omega}$, we obtain the estimate: 
\begin{eqnarray}
|u^{n+1}_h-u_h|^2_{2,\Omega}&\leqslant& \lambda A|u^{n+1}_h-u_h|^2_{2,\Omega}+\lambda B|u^{n+1}_h-u_h|_{2,\Omega}|u_h-u^n_h|_{2,\Omega},
\end{eqnarray}
which implies, $
(1-\lambda A)|u^{n+1}_h-u_h|^2_{2,\Omega}\leqslant\lambda B|u^{n+1}_h-u_h|_{2,\Omega}|u_h-u^n_h|_{2,\Omega}.
$
Simplifying by  $|u^{n+1}_h-u_h|_{2,\Omega}$  we have 
$(1-\lambda A)|u^{n+1}_h-u_h|_{1,\Omega}\leqslant \lambda B|u^{n}_h-u_h|_{2,\Omega}.$
If $(1-\lambda A)>0$ we deduce
$|u^{n+1}_h-u_h|_{2,\Omega}\leqslant \lambda(1-\lambda A)^{-1}B|u^{n}_h-u_h|_{2,\Omega}$.
\end{proof}
\begin{rmq} (Algorithm convergence)
	From Theorem \ref{thm1}
 and by recurrence, 
we obtain estimate:
\begin{eqnarray}
|u^{n}_h-u_h|_{2,\Omega}\leqslant \left[\lambda(1-\lambda A)^{-1}B\right]^n|u^{0}_h-u_h|_{2,\Omega}, \forall n\in \mathbb{N}.
\end{eqnarray}
Consequently, we have the convergence of the sequence $(u^n_h)_{n\in \N}$ in the norm $|\cdot|_{2,\Omega}$ if 
\begin{equation}\label{cond}
0< \lambda(1-\lambda A)^{-1}B< 1.
\end{equation}
 The condition (\ref{cond}) is often satisfied in pratice by making a appropriate choice of $\Omega$, $p$ and $\lambda$ \cite{Jad}.
\end{rmq}
\section{Errors analysis}\label{analyse}
\subsection{A priori error analysis}\label{apriori}
For the convergence of the employed numerical method, we define the  interpolation operator and its approximation properties.\\
Let $(K, P_K, \Sigma_{K})$ be a finite element. For all $v\in C^1$, $P_K$ interpolated by $v$ is the unique element $\Pi_hv \in P_K$ such that
\begin{eqnarray}
\hspace*{3.5cm} \varphi_i(\Pi_hv)=\varphi_i(v), \forall i=1,2,\cdots,N \text{ where } N=\dim  P_K
\end{eqnarray}
\begin{itemize}
\item In the case of Bell \cite{MM}: Let $\hat{\Pi_h}^{\small{Bell}}$ be the interpolation operator associated with the Bell triangle. We have:
	\begin{equation}
	\begin{array}{llll}
	\displaystyle\hat{\Pi_h}^{\small{Bell}}\hat{v}=\sum_{j=1}^{6}\sum_{i=1}^{3}\Big[\hat{v}(\hat{a_i})\phi_{i,j}+\dfrac{\partial \hat{v}}{\partial \hat{x}}(\hat{a_i})\phi_{i,j}+\dfrac{\partial \hat{v}}{\partial \hat{y}}(\hat{a_i})\phi_{i,j}\\\\+\dfrac{\partial^2 \hat{v}}{\partial \hat{x}^2}(\hat{a_i})\phi_{i,j}+\dfrac{\partial^2 \hat{v}}{\partial \hat{x} \partial \hat{y}}(\hat{a_i})\phi_{i,j}+\dfrac{\partial^2 \hat{v}}{\partial \hat{y}^2}(\hat{a_i})\phi_{i,j}\Big].
	\end{array}
	\end{equation}
	\item In the case of HCT-C \cite{PG-Clough,PG-CloughI}: Let $\Pi^{\small{HCT-C}}_h$ be the interpolation operator associated with the complete Hsieh-Clough-Tocher triangle and $\Pi^{\small{HCT-C}}_{h_i}$ the restriction of $\Pi^{\small{HCT-C}}_h$ to the triangle $K_i$. We have:
	\begin{equation}
	\label{}
	\begin{array}{llll}
	\displaystyle\Pi^{\small{HCT-C}}_{h_i}v=\sum_{j=1}^{3}\big[v(a_j)r^0_{i,j}+\nabla v(a_j)\cdot (a_{j+1}-a_j)r^1_{i,j,j+1}\\\\+\nabla v(a_j)\cdot (a_{j+2}-a_j)r^1_{i,j,j+2}+\nabla v(b_j)\cdot (a_j-c_j)r^1_{\bot,i,j}\big].
	\end{array}
	\end{equation}
	\item In the case of HCT-R \cite{PG-Clough,PG-CloughI}: Let $\check{\Pi}^{\small{HCT-R}}_h$ be the interpolation operator associated with the complete Hsieh-Clough-Tocher triangle and $\check{\Pi}^{\small{HCT-R}}_{h_i}$ the restriction of $\check{\Pi}^{\small{HCT-R}}_h$ to the triangle $K_i$. We have:
	\begin{equation}
	\label{}
	\check{\Pi}^{\small{HCT-R}}_{h_i}v=\sum_{j=1}^{3}\left[v(a_j)\check{r}^0_{i,j}+\nabla v(a_j)\cdot (a_{j+1}-a_j)\check{r}^1_{i,j,j+1}+\nabla v(a_j)\cdot (a_{j+2}-a_j)\check{r}^1_{i,j,j+2}\right].
	\end{equation}
\end{itemize}
\blem(\textbf{Approximation Properties})
	 In the case of Bell \cite{MM}, we have the property,  
	\begin{equation}
	\label{}
	\|v-\Pi^{\small{Bell}}_hv\|_{m,K}\leqslant Ch_K^{5-m}\|v\|_{5,K}, \text{ pour } m=0,1,2,3 \text{ et } v\in H^5(K).
	\end{equation}
	 While in the case of HCT-C \cite{PG-Clough,PG-CloughI} the estimate
	\begin{equation}
	\label{EHCTC}
	\|v-\Pi^{\small{HCT-C}}_hv\|_{m,K}\leqslant Ch_K^{4-m}\|v\|_{4,K}, \text{ pour } m=0,1,2 \text{ et } v\in H^4(K).
	\end{equation} is satisfy.
	 In the case of HCT-R \cite{PG-Clough,PG-CloughI}, we obtain, 
	\begin{equation}
	\label{EHCTR}
	\|v-\Pi^{\small{HCT-R}}_hv\|_{m,K}\leqslant Ch_K^{4-m}\|v\|_{4,K}, \text{ pour } m=0,1,2, \text{ et } v\in H^4(K).
	\end{equation}
\elem
\blem\cite{Bendali}\label{Ben} Let $u\in H^4(\Omega)$ unique solution of $(\ref{p})$. Let $G$ define by $(\ref{g})$ and $S$ define by $(\ref{s})$. Then, we have the following estimates : 
\begin{equation}
\label{priori1}
\|u-\Pi_hu\|_V\leqslant C_1h^2,
\end{equation}
\begin{equation}
\label{priori2}
\|(S_h-S)\circ G(u)\|_V\leqslant C_2h^2,
\end{equation}
where $C_1$ and $C_2$ are strictly positive real constants.
\elem
\begin{proof}	The finite elements used for the discretization are HCT and the family of triangulations considered is regular. Using the estimates \ref{EHCTC} and \ref{EHCTR}  we have the estimate  (\ref{priori1}) for $m=2$.
 Moreover, on the estimate $\|(S_h-S)\circ G(u)\|_V\leqslant C_2h^2$ as the discretization error estimate of the biharmonic operator.
\end{proof}
\begin{rmq}
	To show the convergence of the method, we use the following a priori error estimate: $
	\|u_h-u\|_V\leqslant M\big(\|u-u_h\|_V + \|(S_h-S)\circ G(u)\|_V\big),$
	and the lemma \ref{Ben} above for obtain $
	\|u_h-u\|_V\leqslant M(C_1h^2 + C_2h^2)$. Hence $
	\|u_h-u\|_V\leqslant \xi h^2 \text{ with } \xi=M(C_1+C_2)> 0$.
\end{rmq}
\begin{rmq}For Bell elements, assuming the exact solution has regularity $H^5(\Omega)$, we obtain a third-order convergence for $m=2$ and to show the convergence of the method, we use the following a priori error estimate: $\|u_h-u\|_V\leqslant M\big(\|u-u_h\|_V + \|(S_h-S)\circ G(u)\|_V\big),$
	and the lemma \ref{Ben} above for obtain $
	\|u_h-u\|_V\leqslant M(C_1h^3 + C_2h^3)$. Hence $
	\|u_h-u\|_V\leqslant \xi h^3 \text{ with } \xi=M(C_1+C_2)> 0$.
\end{rmq}
\subsection{A posteriori error analysis}\label{aposteriori}
In this section, we use the weighted residual method to determine the a posteriori error indicators. We show that the family of a posteriori error indicators obtained is
both reliable and efficient.
We start by defining the a posteriori error indicators and then we do an a posteriori error analysis
based on these indicators.
\bdfn(\textbf{Error indicators}) \label{indicator}
Let  $u^{n+1}_h$ be the unique solution of the iterative
problem $(\ref{al})$.
Then, the a posteriori error indicators are locally defined by: 
\begin{equation}
\eta_{K,n}:=\left[(\eta^{(D)}_{K,n})^2+(\eta^{(L)}_{K,n})^2\right]^{1/2},\ \forall K\in \mathcal{T}_h,
\end{equation}
where 
\begin{eqnarray}\nonumber
\eta^{(D)}_{K,n}&:=& \displaystyle h_{K}^2\left\|f_h-\Delta^2 u^{n+1}_h-\lambda|u^n_h|^{2p}u^n_h\right\|_{L^2(K)}+\displaystyle\sum_{e\in \varepsilon(K)}\left\|\Delta u^{n+1}_h\right\|_{L^2(e)}\\
&+&\displaystyle\frac{1}{2}\sum_{e\in \varepsilon(K)}h^{3/2}_{e}\left\|\left[\frac{\partial(\Delta u^{n+1}_h)}{\partial n}\right]_e\right\|_{L^2(e)}
\end{eqnarray}
and 
\begin{eqnarray}
	\eta^{(L)}_{K,n}:=\lambda h_K^2\left|u^n_h-u^{n+1}_h\right|_{2,K}.
\end{eqnarray}
$\eta^{(D)}_{K,n}$ denotes the discretization error indicator while  $\eta^{(L)}_{K,n}$  is due to linearization.
\edfn
\subsubsection{Reliability of indicators}\label{upper}
	In order to perform a posteriori error analysis based on these indicators, we introduce some notations. 
We note by  $\varepsilon(K)$ the set of
sides of a triangle $K$ of $\mathcal{T}_h$ which are not contained in $\partial\Omega$ and 
 $h_K$  denotes the diameter of an element $K$ of $\mathcal{T}_h$ while $h_e$ is the length an element $e$ of $\varepsilon(K)$.  Approximation of the
data $f$ in $Z_h$ where 
$Z_h=\{g_h\in L^2(\Omega);\forall K\in \mathcal{T}_h,{g_h}_{|_K}\in \mathbb{P}_l(K)\}, \mbox{ } l\in\mathbb{N}$ is denoted by $f_h$ and 
 $\left[\frac{\partial \Delta u_h}{\partial n}\right]_e$ the jump of $\frac{\partial \Delta u_h}{\partial n}$ through the interior edges. We set:
 \begin{equation}\label{form}
 a(u,v):=\int_{\Omega}\Delta u (x)\Delta v (x)dx+\lambda\int_{\Omega}|u^{n}_h (x)|^{2p} v (x)dx.
 \end{equation}
In order to perform an a posteriori error increase based on these indicators, we determine the residue of $u^{n+1}_h$ following $v$ where $v\in V$. Finally, the residue of $u_h$ in $v$ for conforming approximation of formulation $(\ref{form})$, denoted $R^h(u_h,v)$ is defined by
$R^h(u_h,v)=a(u-u_h,v).$ If $u^{n+1}_h$ is the unique solution of the iteratif problem $(\ref{al})$ and let $v\in V:=H_0^2(\Omega)$,
then we can obtained the residu equation:
\begin{eqnarray}
	R^h(u^{n+1}_h,v):=\displaystyle\sum_{K\in \mathcal{T}_h} \left[R_{1,K}+R_{2,K}+R_{3,K}+R_{4,K}+R_{5,K}\right],
\end{eqnarray}
where,
\begin{eqnarray*}
	R_{1,K}&:=& \int_{K}(f-f_h)(v-v_h)dx\\
	R_{2,K}&:=& \displaystyle\int_{K}(f_h-\Delta^2 u^{n+1}_h-\lambda|u^n_h|^{2p}u^n_h)(v-v_h)\\
	R_{3,K}&:=& -\displaystyle\sum_{e\in \varepsilon(K)}\int_{e}\Delta u^{n+1}_h\frac{\partial (v-v_h)}{\partial n}\\
	R_{4,K}&:=& \displaystyle\frac{1}{2}\sum_{e\in \varepsilon(K)}\int_{e}\left[\frac{\partial\Delta u^{n+1}_h}{\partial n}\right](v-v_h)\\
	R_{5,K}&:=&\lambda \int_{K}|u^{n}_h|^{2p}(u^{n}_h-u^{n+1}_h)(v-v_h).
\end{eqnarray*}
In order to increase the overall error $|u-u_h|_{2,\Omega}$ by the a posteriori error indicators define in Definition \ref{indicator}, we need a regularization operator called the Clément interpolation operator.
\bdfn(\textbf{Clément interpolation operator})
Clément's interpolation operator is defined as follows:

\begin{eqnarray}\label{Cl}
C^h : V\to V_h\\
v\mapsto C^hv\nonumber
\end{eqnarray}
with
\begin{eqnarray}
\displaystyle C^hv=\sum_{x\in \mathcal{N}_h(\Omega)}\dfrac{1}{\mbox{Card}(W_x)}
\left(\int_{W_x}v\right)\delta_x,
\end{eqnarray}
where  $\mathcal{N}_h(\Omega)$ is the set of interior vertices, $W_x := \{K\in \mathcal{T}_h : x\in  \mathcal{N}_h(\Omega)\}$.
\edfn
This operator checks certain properties of local approximations .
\bpro\cite{CL}
There exist $c_1$ and $c_2$ strictly positive real constants such that: 
\begin{equation}
\label{cl1}
\forall v\in H^m(\Delta_K), 1\leqslant m\leqslant k+1,
\left\|v-C^hv\right\|_{l,K}\leqslant c_1h^{m-l}_K\|v\|_{H^m(\Delta_K)},
\end{equation}
and 
\begin{equation}
\label{cl2}
\forall v\in H^m(\Delta_e), 1\leqslant m\leqslant k+1,
\left\|v-C^hv\right\|_{l,e}\leqslant c_2h^{m-l-\frac{1}{2}}_K\|v\|_{H^m(\Delta_e)},
\end{equation}
with $\Delta_K=\cup\{K'\in \mathcal{T}_h : K'\cap K\neq \phi\}$ and $\Delta_e=\cup\{K'\in \mathcal{T}_h : K'\cap e\neq \phi\}$.
\epro\mbox{ }\\
In addition, the map
$\gamma_1 : \mathcal{C}^2(\Omega)\mapsto L^2(\Gamma)$ is defined by $\gamma_1(v):=\frac{\partial v}{\partial n}$, that is\\  $\gamma_1(v):=\nabla v\cdot n \mbox{ on } \Gamma$
is linear continuous. Then, for $K\in\cT_h$ its restriction\\
$\gamma_1 : H^2(K)\mapsto H^{3/2}(\partial K)\subset L^2(\partial K)$
is continuous and there exists, $ \alpha >0 $ such that for all $v\in H^2(K)$
$\|\gamma_1(v)\|_{L^2(\partial K)}\leqslant \alpha\|v\|_{H^2(K)}.$
By replacing $v$ by $v-v_h$ and the fact that $e\subset \partial K$  we obtain
$\|\gamma_1(v-v_h)\|_{L^2(e)}\leqslant \alpha\|v-v_h\|_{H^2(K)}.$
Hence,
\begin{equation}
\left\|\dfrac{\partial (v-v_h)}{\partial n}\right\|_{L^2(e)}\leqslant \alpha\|v-v_h\|_{H^2(K)}.
\end{equation}
By taking $v_h=C^hv$ and  $(l,m)=(2,2)$, we obtain:
$\|v-v_h\|_{H^2(K)}\leqslant c_1\|v\|_{H^2(\Delta_K)}$ and 
\begin{equation}
\label{C1}
\left\|\dfrac{\partial (v-v_h)}{\partial n}\right\|_{L^2(e)}\leqslant c_2\alpha h_K^{-1/2} \|v\|_{H^2(\Delta_K)}.
\end{equation}
While, for $ (m,l)=(2,0)$ we have:
\begin{equation}
\label{C2}
\|v-C^hv\|_{L^2(K)}\leqslant c_1h^{2}_K\|v\|_{H^2(\Delta_K)}.
\end{equation}
\begin{equation}
\label{C3}
\|v-C^hv\|_{L^2(e)}\leqslant c_2h^{3/2}_K\|v\|_{H^2(\Delta_e)}.
\end{equation}
where $c_1, c_2$  are constants independent of $h$. 
The reliability of the family $(\eta_K)_{K\in \mathcal{T}_h}$ is finally justified by the following
theorem: 
\bthm (\textbf{Upper error bound})
 Let $u\in V$ be the exact solution of the nonlinear problem $(\ref{p})$ and let $u^{n+1}_h\in V_h$ be its approximation in the sense of finite elements for the iterative algorithm $(\ref{al})$.
Then there exists a strictly positive real constant  $C_{rel}$ independant of $n$, such that: 
\begin{equation}
\label{in}
\|u-u^{n+1}_h\|_V\leqslant \displaystyle C_{rel}\left[\left(\sum_{K\in \mathcal{T}_h}\eta^{2}_{K,n}\right)^{1/2}+\left(\sum_{K\in \mathcal{T}_h}h^4_K\|f-f_h\|^2_{L^2(K)}\right)^{1/2}\right],
\end{equation}
where  $\left\{\eta_{K,n}\right\}_{K\in \cT_h}$ is defined in 
definition \ref{indicator}.
\ethm
\begin{proof}
For the proof of this theorem, we recall the residue.
\begin{eqnarray*}
R^h(u^{n+1}_h,v)=\displaystyle\sum_{K\in \mathcal{T}_h}\Big(\int_{K}(f-f_h)(v-v_h)dx\\
+\displaystyle\int_{K}(f_h-\Delta^2 u^{n+1}_h-\lambda|u^n_h|^{2p}u^n_h)(v-v_h)dx\\
-\displaystyle\sum_{e\in \varepsilon(K)}\int_{e}\Delta u^{n+1}_h\frac{\partial(v-v_h)}{\partial n}ds
+\displaystyle\dfrac{1}{2}\sum_{e\in \varepsilon(K)}\int_{e}[\frac{\partial\Delta u^{n+1}_h}{\partial n}](v-v_h)ds\\
+\lambda \int_{K}|u^{n}_h|^{2p}(u^{n}_h-u^{n+1}_h)(v-v_h)dx\Big),
\end{eqnarray*}
which implies,
\begin{eqnarray*}
|R^h(u^{n+1}_h,v)|\leqslant\displaystyle\sum_{K\in \mathcal{T}_h}\Big(\int_{K}|f-f_h||v-v_h|dx\\
+\displaystyle\int_{K}|f_h-\Delta^2 u^{n+1}_h-\lambda|u^n_h|^{2p}u^n_h||v-v_h|dx\\
+\displaystyle\sum_{e\in \varepsilon(K)}\int_{e}|\Delta u^{n+1}_h|\left|\frac{\partial(v-v_h)}{\partial n}\right|ds
+\displaystyle\dfrac{1}{2}\sum_{e\in \varepsilon(K)}\int_{e}\left|\left[\frac{\partial\Delta u^{n+1}_h}{\partial n}\right]\right||v-v_h|ds\\+\lambda \int_{K}|u^{n}_h|^{2p}|u^{n}_h-u^{n+1}_h||v-v_h|dx\Big).
\end{eqnarray*}
The Cauchy-Schwarz inequality leads to
\begin{eqnarray*}
|R^h(u^{n+1}_h,v)|\leqslant\displaystyle\sum_{K\in \mathcal{T}_h}\Big(\|f-f_h\|_{L^2(K)}\|v-v_h\|_{L^2(K)}\\
+\|f_h-\Delta^2 u^{n+1}_h-\lambda|u^n_h|^{2p}u^n_h\|_{L^2(K)}\|v-v_h\|_{L^2(K)}\\
+\displaystyle\sum_{e\in \varepsilon(K)}\|\Delta u^{n+1}_h\|_{L^2(e)}\left\|\frac{\partial(v-v_h)}{\partial n}\right\|_{L^2(e)}\\
+\displaystyle\frac{1}{2}\sum_{e\in \varepsilon(K)}\left\|\left[\frac{\partial \Delta u^{n+1}_h}{\partial n}\right]\right\|_{L^2(e)}\|v-v_h\|_{L^2(e)}\\
+\lambda c^{2p}_1\|f\|^{2p}_{-2,K}\|u^{n}_h-u^{n+1}_h\|_{L^2(K)}\|v-v_h\|_{L^2(K)}\Big).
\end{eqnarray*}
Taking $v_h$ as Clément's interpolated we obtain,
\begin{eqnarray*}
|R^h(u^{n+1}_h,v)|\leqslant\displaystyle\sum_{K\in \mathcal{T}_h}\Big(\|f-f_h\|_{L^2(K)}\|v-C^hv\|_{L^2(K)}\\
+\|f_h-\Delta^2 u^{n+1}_h-\lambda|u^n_h|^{2p}u^n_h\|_{L^2(K)}\|v-C^hv\|_{L^2(K)}\\
+\displaystyle\sum_{e\in \varepsilon(K)}\|\Delta u^{n+1}_h\|_{L^2(e)}\left\|\frac{\partial(v-v_h)}{\partial n}\right\|_{L^2(e)}\\
+\displaystyle\frac{1}{2}\sum_{e\in \varepsilon(K)}\left\|\left[\frac{\partial \Delta u^{n+1}_h}{\partial n}\right]\right\|_{L^2(e)}\|v-C^hv\|_{L^2(e)}\\
+\lambda c^{2p}_1\|f\|^{2p}_{-2,K}\|u^{n}_h-u^{n+1}_h\|_{L^2(K)}\|v-C^hv\|_{L^2(K)}\Big).
\end{eqnarray*}
Subsequently using the inequalities defined in \eqref{C1}, \eqref{C2} and \eqref{C3} we get,
\begin{eqnarray*}
|R^h(u^{n+1}_h,v)|\leqslant\displaystyle\sum_{K\in \mathcal{T}_h}\Big(\alpha_1h^2_K\|f-f_h\|_{L^2(K)}\|v\|_{H^2(\Delta_K)}\\
+\alpha_1h^2_K\|f_h-\Delta^2 u^{n+1}_h-\lambda|u^n_h|^{2p}u^n_h\|_{L^2(K)}\|v\|_{H^2(\Delta_K)}\\
+\displaystyle\alpha c_3\sum_{e\in \varepsilon(K)}\|\Delta u^{n+1}_h\|_{L^2(e)}\|v\|_{H^2(\Delta_K)}\\
+\displaystyle\frac{1}{2}\sum_{e\in \varepsilon(K)}c_4h^{3/2}_e\left\|\left[\frac{\partial \Delta u^{n+1}_h}{\partial n}\right]\right\|_{L^2(e)}\|v\|_{H^2(\Delta_e)}\\
+\lambda \alpha_1c^{2p}_1\|f\|^{2p}_{-2,K}h^2_K\|u^{n}_h-u^{n+1}_h\|_{L^2(K)}\|v\|_{H^2(\Delta_K)}\Big).
\end{eqnarray*}
Sinve  $\|v\|_{H^2(\Delta_K)}\leqslant \|v\|_{H^2(\Omega)}$ and $\|v\|_{H^2(\Delta_e)}\leqslant \|v\|_{H^2(\Omega)}$,
furthermore, $\|\cdot\|_{H^2(\Omega)}$ is equivalent to $\|\cdot\|_V$ \ on\  $H^2_0(\Omega)$.
Therefore,
\begin{eqnarray*}
|R^h(u^{n+1}_h,v)|\leqslant\displaystyle\sum_{K\in \mathcal{T}_h}\Big[\alpha_1h^2_K\|f-f_h\|_{L^2(K)}\\
+\alpha_1h^2_K\|f_h-\Delta^2 u^{n+1}_h-\lambda|u^n_h|^{2p}u^n_h\|_{L^2(K)}\\
+\displaystyle\alpha c_3\sum_{e\in \varepsilon(K)}\|\Delta u^{n+1}_h\|_{L^2(e)}\\
+\displaystyle\frac{1}{2}\sum_{e\in \varepsilon(K)}c_4h^{3/2}_e\|[\frac{\partial \Delta u^{n+1}_h}{\partial n}]\|_{L^2(e)}\\
+\lambda \alpha c'_2c^{2p}_1\|f\|^{2p}_{-2,K}h^2_K|u^{n}_h-u^{n+1}_h|_{2,K}\Big]\|v\|_V,
\end{eqnarray*}
which lead to
\begin{eqnarray*}
\dfrac{|R^h(u^{n+1}_h,v)|}{\|v\|_V}\leqslant\displaystyle\sum_{K\in \mathcal{T}_h}\Big[\alpha_1h^2_K\|f-f_h\|_{L^2(K)}\\
+\alpha_1h^2_K\|f_h-\Delta^2 u^{n+1}_h-\lambda|u^n_h|^{2p}u^n_h\|_{L^2(K)}\\
+\displaystyle\alpha c_3\sum_{e\in \varepsilon(K)}\|\Delta u^{n+1}_h\|_{L^2(e)}\\
+\displaystyle\frac{1}{2}\sum_{e\in \varepsilon(K)}c_4h^{3/2}_e\|[\frac{\partial \Delta u^{n+1}_h}{\partial n}]\|_{L^2(e)}\\
+\lambda \alpha c'_2c^{2P}_1\|f\|^{2p}_{-2,K}h^2_K|u^{n}_h-u^{n+1}_h|_{2,K}\Big].
\end{eqnarray*}

By setting $C= \max\{\alpha_1c^{2p}_1\|f\|^{2p}_{-2,\Omega},\alpha_1, \alpha_1c_3 , c_4\}> 0$ we obtain: 
\begin{eqnarray}
\dfrac{|R^h(u^{n+1}_h,v)|}{\|v\|_V}\leqslant\displaystyle C\sum_{K\in \mathcal{T}_h}\Big[h^2_K\|f-f_h\|_{L^2(K)}\\\\
+h^2_K\|f_h-\Delta^2 u^{n+1}_h-\lambda|u^n_h|^{2p}u^n_h\|_{L^2(K)}
+\displaystyle\sum_{e\in \varepsilon(K)}\|\Delta u^{n+1}_h\|_{L^2(e)}\\\\
+\displaystyle\frac{1}{2}\sum_{e\in \varepsilon(K)}h^{3/2}_e\|[\frac{\partial \Delta u^{n+1}_h}{\partial n}]\|_{L^2(e)}
+\lambda h^2_K|u^{n}_h-u^{n+1}_h|_{2,K}\Big].
\end{eqnarray}
Thus, the set $\left\{\dfrac{|R^h(u^{n+1}_h,v)|}{\|v\|_V} :v\in V^*\right\}$ vis non-empty and increased by $\R$. Therefore,
\begin{eqnarray*}
\sup_{v\in V^*}\dfrac{|R^h(u^{n+1}_h,v)|}{\|v\|_V}\leqslant \displaystyle C\sum_{K\in \mathcal{T}_h}\Big[
h^2_K\|f_h-\Delta^2 u^{n+1}_h-\lambda|u^n_h|^{2p}u^n_h\|_{L^2(K)}\\
+\displaystyle\sum_{e\in \varepsilon(K)}\|\Delta u^{n+1}_h\|_{L^2(e)}
+\displaystyle\frac{1}{2}\sum_{e\in \varepsilon(K)}h^{3/2}_e\|[\frac{\partial \Delta u^{n+1}_h}{\partial n}]\|_{L^2(e)}\\+h^2_K\|f-f_h\|_{L^2(K)}
+\lambda h^2_K|u^{n}_h-u^{n+1}_h|_{2,K}\Big].
\end{eqnarray*}
We deduce that,
\begin{equation*}
\label{}
\|u-u^{n+1}_h\|_V\leqslant \displaystyle C\sum_{K\in \mathcal{T}_h}\big(\eta^{(D)}_K+h^2_k\|f-f_h\|_{L^2(K)}+\eta^{(L)}_K\big),
\end{equation*}
with 
\begin{eqnarray*}
\eta^{(D)}_{K,n}&=& h^2_K\|f_h-\Delta^2 u^{n+1}_h-\lambda|u^n_h|^{2p}u^n_h\|_{L^2(K)}
+\displaystyle\sum_{e\in \varepsilon(K)}\|\Delta u^{n+1}_h\|_{L^2(e)}\\
&+&\displaystyle\frac{1}{2}\sum_{e\in \varepsilon(K)}h^{3/2}_e\left\|[\frac{\partial \Delta u^{n+1}_h}{\partial n}]\right\|_{L^2(e)} \mbox{ and }
\\
\eta^{(L)}_{K,n}&=&\lambda h^2_K|u^n_h-u^{n+1}_h|_{2,K}
\end{eqnarray*}
Using the Cauchy-Schwarz inequality twice we have : 
\begin{equation*}
\|u-u^{n+1}_h\|_V\leqslant \displaystyle C_{rel}\left[\left(\sum_{K\in \mathcal{T}_h}\eta^{2}_{K,n}\right)^{1/2}+\left(\sum_{K\in \mathcal{T}_h}h^4_K\|f-f_h\|^2_{L^2(K)}\right)^{1/2}\right],
\end{equation*}
with $C_{rel}=C\left[Card(\mathcal{T}_h)\right]^{1/2}> 0$.
\end{proof}
\brmq
From there, we see that the error is increased by a calculable quantity. So, the error is controlled and we talk about reliability.
\ermq
\subsubsection{Optimality of indicators}\label{lower}
We show in this subsection that the family of a posteriori error indicators $(\eta_{K,n})_{K\in \mathcal{T}_{h}}$ forms a good error map. For this we will need some tools.

\bdfn (\textbf{Bubble functions}) 
The bubble function $\psi_K$ on a mesh $K$ is an element of $\mathbb{P}_{d+1}(K)$ is defined by :
\begin{equation}
\psi_K :=(d+1)^{d+1}\times\displaystyle\prod_{i=1}^{d+1}\lambda_{i,K},
\end{equation}
where $\lambda_{i,K}$ are barycentric coordinate functions associated with $K$.
If $e\in \varepsilon(K)$, we define $\psi_{e,K}\in \mathbb{P}_d(K)$ by :
$
\psi_{e,K} :=(d)^{d}\times\displaystyle\prod_{i=1}^{d}\lambda_{x_i,K}$
and we then define $\psi_e$ on $W_e:=K_1\cup K_2$ where $e=\partial K_1\cap \partial K_2$ by :
$
\psi_{e/K_i} :=\psi_{e,K_i}\  \forall i=1,2.$

\edfn
\mbox{ }\\
 Bubble functions check the following properties\cite{W1}, namely,
$\psi_{K}=0 \text{ on } \partial_K$; \\
$\psi_{e}=0 \text{ on } \partial{W_e}$;
$0\leqslant\psi_{K}\leqslant1$;
$0\leqslant\psi_{e}\leqslant1$ and
$\|\psi_{K}\|_{L^{\infty}(K)}=\|\psi_{e}\|_{L^{\infty}(W_e)}=1.$
The fact that the bubble functions are zero on $\partial K$ or on $\partial {W_e}$ makes it possible to
cancel the edge terms involved in the calculations.

\bdfn\cite{christine} (\textbf{Extension operator}) We define the raising or extension operator as follows :\\
$\begin{array}{cccc}
\mathcal{L}_{e,K} : &C^0(e)&\to&C^0(\bar{K}),\\
&v&\mapsto&\mathcal{L}_{e,K}v, 
\end{array}$
such that  $\mathcal{L}_{e,K}v|_e=v$.
\edfn
These bubble and extension operator functions check the following inequalities.
\blem(\textbf{Inverse inequalities})\cite{Verf}
Let  $(\mathcal{T}_h)_{h>0}$ be a regular triangulation family on $\bar{\Omega}$. Then, for all $v_K\in \mathbb{P}_{k_0}(K)$ and $v_e\in \mathbb{P}_{k_1}(e)$ with $e\in \varepsilon(K)$, we have the following equivalences and inequalities \cite{W1}:
There exist strictly positive real constants $c_1$ et $c_2$ independent of the $v_K$ and $v_e$ functions such that :
\begin{equation}
\label{inv}
\begin{array}{cccc}
& c_1\|v_K\|_{L^2(K)}\leqslant
\|v_K\psi^{1/2}_K\|_{L^2(K)}\leqslant c_2\|v_K\|_{L^2(K)}.\\\\
& c_1\|v_e\|_{L^2(e)}\leqslant
\|v_e\psi^{1/2}_e\|_{L^2(e)}\leqslant c_2\|v_e\|_{L^2(e)}.\\\\
&\|\nabla^l(v_K)\|_{L^2(K)}\leqslant c_1h^{-l}_K\|v_K\|_{L^2(K)}.\\\\
&\|\nabla^l(v_K\psi_K)\|_{L^2(K)}\leqslant h^{-l}_K\|v_K\|_{L^2(K)}.\\\\
&\|\nabla^{l'}(v_e\psi_e)\|_{L^2(e)}\leqslant h^{-l'+1/2}_K\|v_e\|_{L^2(e)}.\\\\
&\|\mathcal{L}_{e,K}(v_e) \psi_e\|_{L^2(K)}\leqslant h^{1/2}_e\|v_e\|_{L^2(e)}.\\\\
&\|\nabla^l(\mathcal{L}_{e,K}(v_e)\psi_e)\|_{L^2(K)}\leqslant h^{1/2-l}_e\|v_e\|_{L^2(e)}.
\end{array}
\end{equation}
\elem

In this subsection we give a reduction of the a posteriori error indicators of discretization and linearization.
\bthm (\textbf{Lower error bound}) Let $u\in V$ be the exact solution of the nonlinear problem $(\ref{p})$, $u^{n+1}_h\in V_h$ be the solution of the iterative problem $(\ref{al})$. Then, there exists a strictly positive real constant 
$C_{eff}$ such that
for each  $K\in \mathcal{T}_h$, we have the following estimates:
\begin{eqnarray*}
\eta^{(D)}_{K,n}\leqslant C_{eff} \left(|u-u_h^{n+1}|_{2,W_K}+\displaystyle\sum_{\kappa\subset W_K}h^2_{\kappa}\left(\|f-f_h\|_{L^2(\kappa)}+\lambda |u_h^n-u_h^{n+1}|_{2,\kappa}\right)\right),
\end{eqnarray*}
and 
\begin{eqnarray*}
\eta^{(L)}_{K,n}\leqslant \lambda h^2_K\left(|u-u_h^{n+1}|_{2,K}+|u-u_h^n|_{2,K}\right).
\end{eqnarray*}
\ethm
\begin{proof}
For the proof, we will first increase each of the terms of the discretization error indicator and then give an
increase of the local discretization error indicator.
To do this, we recall the residual equation.
\begin{eqnarray}\label{residu}\nonumber
R^h(u^{n+1}_h,v)=\displaystyle\sum_{K\in \mathcal{T}_h}\Big(\int_{K}(f-f_h)(v)dx\\\nonumber
+\displaystyle\int_{K}(f_h-\Delta^2 u^{n+1}_h-\lambda|u^n_h|^{2p}u^n_h)(v)dx\\
-\displaystyle\sum_{e\in \varepsilon(K)}\int_{e}\Delta u^{n+1}_h\frac{\partial(v)}{\partial n}ds
+\displaystyle\dfrac{1}{2}\sum_{e\in \varepsilon(K)}\int_{e}[\frac{\partial\Delta u^{n+1}_h}{\partial n}](v)ds\\\nonumber
+\lambda \int_{K}|u^{n}_h|^{2p}(u^{n}_h-u^{n+1}_h)(v)dx\Big).
\end{eqnarray}
We bound each term of the residual separately.
\begin{enumerate}
	\item 
For all $K\in \mathcal{T}_h$, we take $v$ equal to $v_K$ with \\
$v_K=\begin{cases}
& (f_h-\Delta^2 u^{n+1}_h-\lambda|u^n_h|^{2p}u^n_h)\psi_K\ on \ K\\
&  0\ on\ \Omega\smallsetminus K,
\end{cases}
$\\
where $\psi$ is the bubble
function on $K$.
We get
\begin{eqnarray*}
R^h(u^{n+1}_h,v_K)=\int_{K}(f-f_h)v_Kdx
+\displaystyle\int_{K}(f_h-\Delta^2 u^{n+1}_h-\lambda|u^n_h|^{2p}u^n_h)v_Kdx\\
+\lambda \int_{K}|u^{n}_h|^{2p}(u^{n}_h-u^{n+1}_h)v_Kdx.
\end{eqnarray*}
We have: 
\begin{eqnarray*}
\int_K(f_h-\Delta^2 u^{n+1}_h-\lambda|u^n_h|^{2p}u^n_h)^2\psi_Kdx = R^h(u^{n+1}_h,v_K)-\int_K(f-f_h)v_Kdx\\
- \lambda\int_K|u^n_h|^{2p}(u^n_h-u^{n+1}_h)v_Kdx \\
\int_K(f_h-\Delta^2 u^{n+1}_h-\lambda|u^n_h|^{2p}u^n_h)^2\psi_Kdx=\int_K\Delta(u-u^{n+1}_h)\cdot\Delta v_Kdx\\
-\int_K(f-f_h)v_Kdx
- \lambda\int_K|u^n_h|^{2p}(u^n_h-u^{n+1}_h)v_Kdx. 
\end{eqnarray*}
\begin{eqnarray*}
\|(\Delta u^{n+1}_h+f_h-\lambda|u^n_h|^{2p})\psi^{1/2}_K\|^2_{L^2(K)}\leqslant \int_K|\Delta(u-u^{n+1}_h)\Delta v_K|dx\\+\int_K|(f-f_h)v_K|dx
+ \lambda\int_K|u^n_h|^{2p}|(u^n_h-u^{n+1}_h)||v_K|dx.
\end{eqnarray*}
\begin{eqnarray*}
\|(f_h-\Delta^2 u^{n+1}_h+f_h-\lambda|u^n_h|^{2p})\psi^{1/2}_K\|^2_{L^2(K)} \leqslant  |u-u^{n+1}_h|_{2,K}|v_K|_{2,K}\\\\+||f-f_h||_{L^2(K)}||v_K||_{L^2(K)}\\
+\lambda c'_2c^{2p}_1\|f\|_{-2,K}|u^n_h-u^{n+1}_h|_{2,K}\|v_K\|_{L^2(K)}.\\
\end{eqnarray*}
As $\psi_K\leqslant 1, c_1\|v_K\|_{L^2(K)}\leqslant
\|v_K\psi^{1/2}_K\|_{L^2(K)}$ and  $\|\nabla^2(v_K\psi_K)\|_{L^2(K)}\leqslant h^{-2}_K\|v_K\|_{L^2(K)}.$
we have : 
\begin{eqnarray*}
\xi^2_2\|(f_h-\Delta^2 u^{n+1}_h-\lambda|u^n_h|^{2p})\|^2_{L^2(K)}\leqslant  \Big(\xi_1 h^{-2}_K|u-u^{n+1}_h|_{2,K}+\|f-f_h\|_{L^2(K)}\\\\
+\lambda c'_2c^{2p}_1\|f\|_{-2,K}|u^n_h-u^{n+1}_h|_{2,K}\Big)\|(f_h-\Delta^2 u^{n+1}_h-\lambda|u^n_h|^{2p})\|_{L^2(K)}.
\end{eqnarray*}
By setting $A'=c'_2c^{2p}_1\|f\|_{-2,K}> 0$ , by simplifying and multiplying by $h^2_K$ we have :
\begin{eqnarray*}
h^2_K\|f_h-\Delta^2 u^{n+1}_h-\lambda|u^n_h|^{2p}\|_{L^2(K)}\leqslant  \dfrac{\xi_1}{\xi^2_2}|u-u^{n+1}_h|_{2,K}+\dfrac{1}{\xi^2_2}h^2_K\|f-f_h\|_{L^2(K)}\\+\lambda \dfrac{1}{\xi^2_2}A'|u^n_h-u^{n+1}_h|.
\end{eqnarray*}
By multiplying the inequality by $h_K$ by setting $\xi =\max\left\{\dfrac{\xi_1}{\xi^2_2},\dfrac{1}{\xi^2_2},\dfrac{1}{\xi^2_2}A'\right\}$ we obtain:
\begin{eqnarray*}
h^2_K\|f_h-\Delta^2 u^{n+1}_h-\lambda|u^n_h|^{2p}\|_{L^2(K)}\leqslant  \xi\Big(|u-u^{n+1}_h|_{2,K}+h^2_K\|f-f_h\|_{L^2(K)}\\
+\lambda h^2_K|u^n_h-u^{n+1}_h|\Big).
\end{eqnarray*}
Thus, the first term of $\eta^{(D)}_{K,n}$ is increased.
\item  Let us major the second term of $\eta^{(D)}_{K,n}$. 
For all $K$ in $\mathcal{T}_h$ and $e\in \varepsilon(K)$, we set $W_e=K\cup K'$ such that $e=\partial K \cap \partial K'$.
We replace in the residue equation the $v$ by $v_e$ with 
$$v_e=\begin{cases}
\mathcal{L}_{e,\kappa}\left([\dfrac{\partial\Delta u^{n+1}_h}{\partial n}]\psi_{e}\right) \text{ on } \kappa\in \{K,K'\}\\
0 \text{ on } \Omega\smallsetminus W_e.
\end{cases}$$
where $\psi_{e}$ denotes the bubble function on $e$.

\begin{eqnarray*}
R^h(u^{n+1}_h,v)=\int_{W_e}(f-f_h)v_edx
+\displaystyle\int_{W_e}(f_h-\Delta^2 u^{n+1}_h-\lambda|u^n_h|^{2p}u^n_h)v_edx\\
-\int_{e}\Delta u^{n+1}_h\frac{\partial(v_e)}{\partial n}ds
+\frac{1}{2}\int_{e}[\frac{\partial\Delta u^{n+1}_h}{\partial n}]v_eds+\lambda \int_{W_e}|u^{n}_h|^{2p}(u^{n}_h-u^{n+1}_h)v_edx.
\end{eqnarray*}

\begin{eqnarray*}
\frac{1}{2}\int_{e}[\frac{\partial\Delta u^{n+1}_h}{\partial n}]v_eds =-R^h(u^{n+1}_h,v)+\int_{W_e}(f-f_h)v_edx\\
+\displaystyle\int_{W_e}(f_h-\Delta^2 u^{n+1}_h-\lambda|u^n_h|^{2p}u^n_h)v_edx\\
-\int_{e}\Delta u^{n+1}_h\frac{\partial(v_e)}{\partial n}ds
+\lambda \int_{W_e}|u^{n}_h|^{2p}(u^{n}_h-u^{n+1}_h)v_edx.
\end{eqnarray*}

\begin{eqnarray*}
\frac{1}{2}\int_{e}[\frac{\partial\Delta u^{n+1}_h}{\partial n}]v_eds =-\int_{W_e}\Delta(u-u^{n+1}_h)\cdot\Delta v_edx+\int_{W_e}(f-f_h)v_edx\\
+\displaystyle\int_{W_e}(f_h-\Delta^2 u^{n+1}_h-\lambda|u^n_h|^{2p}u^n_h)v_edx\\
-\int_{e}\Delta u^{n+1}_h\frac{\partial(v_e)}{\partial n}ds
+\lambda \int_{W_e}|u^{n}_h|^{2p}(u^{n}_h-u^{n+1}_h)v_edx.
\end{eqnarray*}

\begin{eqnarray*}
\frac{1}{2}\int_{e}[\frac{\partial\Delta u^{n+1}_h}{\partial n}]^2\psi_eds \leqslant\|\Delta(u-u^{n+1}_h)\|_{L^2(W_e)}\|\Delta v_e\|_{L^2(W_e)}\\+\|f-f_h\|_{L^2(W_e)}\|v_e\|_{L^2(W_e)}\\
+\|f_h-\Delta^2 u^{n+1}_h-\lambda|u^n_h|^{2p}u^n_h\|_{L^2(W_e)}\|v_e\|_{L^2(W_e)}\\
+\|\Delta u^{n+1}_h\|_{L^2(e)}\|\frac{\partial(v_e)}{\partial n}\|_{L^2(e)}
+\lambda c_1^{2p}|f|^{2p}_{-2,W_e}\|u^{n}_h-u^{n+1}_h\|_{L^2(W_e)}\|v_e\|_{L^2(W_e)}.
\end{eqnarray*}

\begin{eqnarray*}
\frac{1}{2}\int_{e}[\frac{\partial\Delta u^{n+1}_h}{\partial n}]^2\psi_eds \leqslant|u-u^{n+1}_h|_{2,W_e}\|\nabla^2(\mathcal{L}_{e,\kappa}([\dfrac{\partial\Delta u^{n+1}_h}{\partial n}]\psi_{e})) \|_{L^2(W_e)}\\\\+\|f-f_h\|_{L^2(W_e)}\|\mathcal{L}_{e,\kappa}([\dfrac{\partial\Delta u^{n+1}_h}{\partial n}]\psi_{e})\|_{L^2(W_e)}\\\\
+\|f_h-\Delta^2 u^{n+1}_h-\lambda|u^n_h|^{2p}u^n_h\|_{L^2(W_e)}\|\mathcal{L}_{e,\kappa}([\dfrac{\partial\Delta u^{n+1}_h}{\partial n}]\psi_{e})\|_{L^2(W_e)}\\\\
+\|\Delta u^{n+1}_h\|_{L^2(e)}\|\frac{\partial}{\partial n}([\dfrac{\partial(\Delta u^{n+1}_h)}{\partial n}]\psi_{e})\|_{L^2(e)}\\\\
+\lambda c_1^{2p}|f|^{2p}_{-2,W_e}\|u^{n}_h-u^{n+1}_h\|_{L^2(W_e)}\|v_e\|_{L^2(W_e)}.
\end{eqnarray*}
Moreover $\left\|\tfrac{\partial }{\partial n}([\dfrac{\partial (\Delta u^{n+1}_h)}{\partial n}]\psi_{e})\right\|_{L^2(e)}=\left(\displaystyle\int_{e}|\nabla([\dfrac{\partial(\Delta u^{n+1}_h)}{\partial n}]\psi_{e})\cdot n|^2ds\right)^{1/2}$\\
and $\left|\nabla([\dfrac{\partial(\Delta u^{n+1}_h)}{\partial n}]\psi_{e})\cdot n\right|^2\leqslant \left\|\nabla([\dfrac{\partial(\Delta u^{n+1}_h)}{\partial n}]\psi_{e})\right\|^2_{L^2(e)}\|n\|^2_{\R^d}$.\\
We have :\\ $\left(\int_{e}|\nabla([\dfrac{\partial(\Delta u^{n+1}_h)}{\partial n}]\psi_{e})\cdot n|^2ds\right)^{1/2}\leqslant (mes(e))^{1/2}\|\nabla([\dfrac{\partial(\Delta u^{n+1}_h)}{\partial n}]\psi_{e})\|_{L^2(e)}\|n\|_{\R^d}$.\\
As $n$ is unitary then $\|n\|_{\R^d}=1$ and further $mes(e)=h_e$ we obtain:\\
$\left\|\tfrac{\partial }{\partial n}([\dfrac{\partial (\Delta u^{n+1}_h)}{\partial n}]\psi_{e})\right\|_{L^2(e)}\leqslant h_e^{1/2}\left\|\nabla([\dfrac{\partial(\Delta u^{n+1}_h)}{\partial n}]\psi_{e})\right\|_{L^2(e)}$.\\\\
According to the inverse inequalities for the extension operator we obtain:
\begin{equation*}
\begin{array}{cccc}
 &\left\|\mathcal{L}_{e,K}\left([\dfrac{\partial (\Delta u^{n+1}_h)}{\partial n}]\psi_{e}\right)\right\|_{L^2(K)}\leqslant c'_1h^{1/2}_e\left\|[\dfrac{\partial (\Delta u^{n+1}_h)}{\partial n}]\right\|_{L^2(e)}.\\\\
 &\left\|\nabla(\mathcal{L}_{e,K}\left([\dfrac{\partial (\Delta u^{n+1}_h)}{\partial n}]\psi_{e})\right)\right\|_{L^2(K)}\leqslant c''_1h^{-1/2}_e\left\|[\dfrac{\partial (\Delta u^{n+1}_h)}{\partial n}]\right\|_{L^2(e)}.\\\\
 &\left\|\nabla^2(\mathcal{L}_{e,K}\left([\dfrac{\partial (\Delta u^{n+1}_h)}{\partial n}]\psi_{e})\right)\right\|_{L^2(K)}\leqslant c_1h^{-3/2}_e\left\|[\dfrac{\partial (\Delta u^{n+1}_h)}{\partial n}]\right\|_{L^2(e)}.
\end{array}
\end{equation*}
Thank to the inverse inequality on $\psi_e$, there is a real constant $c'_1>0$ such that
$$c_1\left\|[\tfrac{\partial \Delta u^{n+1}_h}{\partial n}]\right\|_{L^2(e)}\leqslant\left\|[\tfrac{\partial\Delta u^{n+1}_h}{\partial n}]\psi^{1/2}_e\right\|_{L^2(e)}.$$

Finally, by multiplying by $h^{3/2}_e$ and using the fact that $h_e\leqslant h_K$ we get:\\
$\begin{array}{ccc}
\frac{1}{2}\displaystyle\sum_{e\in\varepsilon_K}h^{3/2}_e\|[\frac{\partial\Delta u^{n+1}_h}{\partial n}]\|_{L^2(e)} \leqslant\displaystyle\sum_{\kappa\in \{K,K'\}}\Big[\xi_1|u-u^{n+1}_h|_{2,\kappa}+\xi'_1h^2_{\kappa}\|f-f_h\|_{L^2(\kappa)}\\\\+\xi'_1\alpha|u-u^{n+1}_h|_{2,\kappa}+\xi'_1\alpha h^2_{\kappa}|f-f_h|_{2,\kappa}\\\\+\xi'_1\alpha\lambda h^2_{\kappa}|u-u^{n+1}_h|_{2,\kappa} +\lambda c_1^{2p}|f|^{2p}_{-2,\kappa}h^2_{\kappa}\|u^{n}_h-u^{n+1}_h\|_{2,\kappa}\Big]\\\\+\xi''_1\displaystyle\sum_{e\in \varepsilon_K}h^{3/2}_e\|\Delta u^{n+1}_h\|_{L^2(e)}.
\end{array}$\\\\

Setting $\xi=\max\left\{\xi_1+\xi'_1\alpha,\xi'_1\alpha+c_1^{2p}|f|^{2p}_{-2,\kappa}\right\}$, we have:
\begin{eqnarray*}
\frac{1}{2}\displaystyle\sum_{e\in\varepsilon_K}h^{3/2}_e\left\|\left[\frac{\partial\Delta u^{n+1}_h}{\partial n}\right]_e\right\|_{L^2(e)} \leqslant\displaystyle\xi\sum_{\kappa\in \{K,K'\}}\Big[|u-u^{n+1}_h|_{2,\kappa}+h^2_{\kappa}\|f-f_h\|_{L^2(\kappa)}\\\\+\lambda h^2_{\kappa}|u-u^{n+1}_h|_{2,\kappa}\Big]+\xi''_1\displaystyle\sum_{e\in \varepsilon_K}h^{3/2}_e\|\Delta u^{n+1}_h\|_{L^2(e)}.
\end{eqnarray*}
\end{enumerate}
We thus obtain an increase of the second term of $\eta^{(D)}_{K,n}$ and the estimation\\
$\eta^{(L)}_{K,n}\leqslant \lambda h_K^2\left(|u-u_h^{n+1}|_{2,K}+|u^n_h-u|_{2,K}\right)$ 
implies the required estimate and finish the
proof.
\end{proof}
\section{Numerical results and discussions}\label{Numerical_results}
In this section, we present numerical results validating the theoretical findings. Figure (Figure \ref{fig:plaque}) gives an illustration of a fixed plate undergoing loads represented by red arrows.
\begin{figure}[H]
	\includegraphics[width=\textwidth]{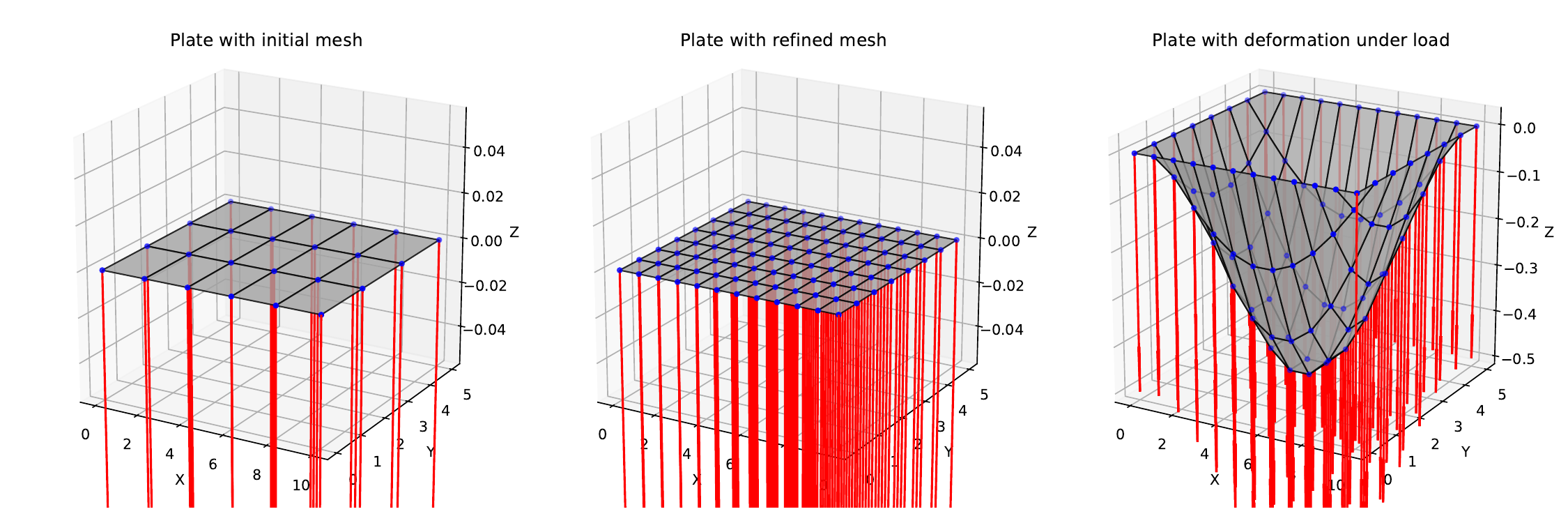} 
	\caption{Plate bending}
	\label{fig:plaque}
\end{figure}
\subsection{Numerical solution}\label{Numerical}
We use the FEniCs software \cite{Fenics}, drawing inspiration from the work presented in the report \cite{Implem}.\
For the unit square domain $\Omega=]0,1[^2$, we generate an initial triangulation (see Figure \ref{fig:Initial_mesh}) and subsequently a Hsieh-Clough-Tocher mesh (see Figure \ref{fig:HCT_mesh}). Considering a synthetic solution $u_e(x,y)=x^2(1-x)^2y^2(1-y)^2$ that satisfies the boundary conditions, we implement the following iterative fixed-point scheme with an initial solution $u^0_h=0.0069$:
\begin{equation}
\label{P_n}
\begin{cases}
\text{ Find } u^{n+1}_h\in V_h \text{ such that},\\
\int_{\Omega}\Delta u^{n+1}_h\Delta v_hdx+\lambda\int_{\Omega}|u^{n}_h|^{2p}u^{n+1}_h v_hdx=\langle f,v_h\rangle \ \forall v_h\in V_h
\end{cases}.
\end{equation} with the classical stopping criterion defined by $err_L = |u^{n+1}_h - u^{n}_h|_{2,\Omega} \leqslant 10^{-7}$.
Figure \ref{fig:Exct_solu} presents the synthetic solution and figure \ref{fig:Numerical_sol} presents the numerical solution after convergence, obtained after a few iterations in a mesh of $384$ cells and $3840$ degrees of freedom with an error in norm $L^2(\Omega)$ equal to $0.00010911603387915538$ and in norm $H^2_0(\Omega)$ equal to $0.005791626721615341$. This result is obtained for $p=1$ and $\lambda=10^4$.
\begin{figure}[H]
	\centering
	\begin{minipage}{0.45\textwidth}
		\centering
		\includegraphics[width=\textwidth]{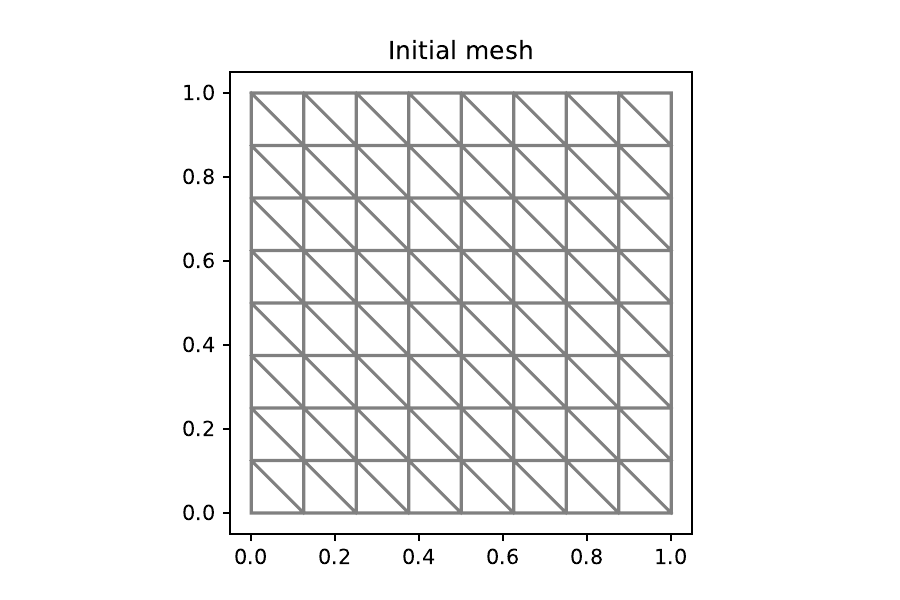}
		\caption{Initial mesh}
		\label{fig:Initial_mesh}
	\end{minipage}\hfill
	\begin{minipage}{0.45\textwidth}
		\centering
		\includegraphics[width=\textwidth]{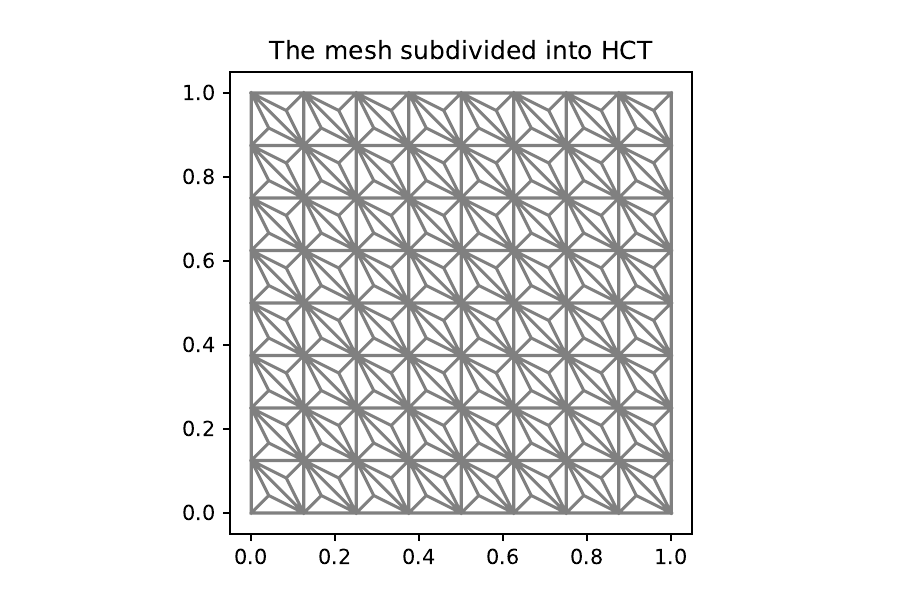}
		\caption{The mesh subdivided into HCT}
		\label{fig:HCT_mesh}
	\end{minipage}
\end{figure}
\begin{figure}[H]
	\centering
	\begin{minipage}{0.45\textwidth}
		\centering
		\includegraphics[width=\textwidth]{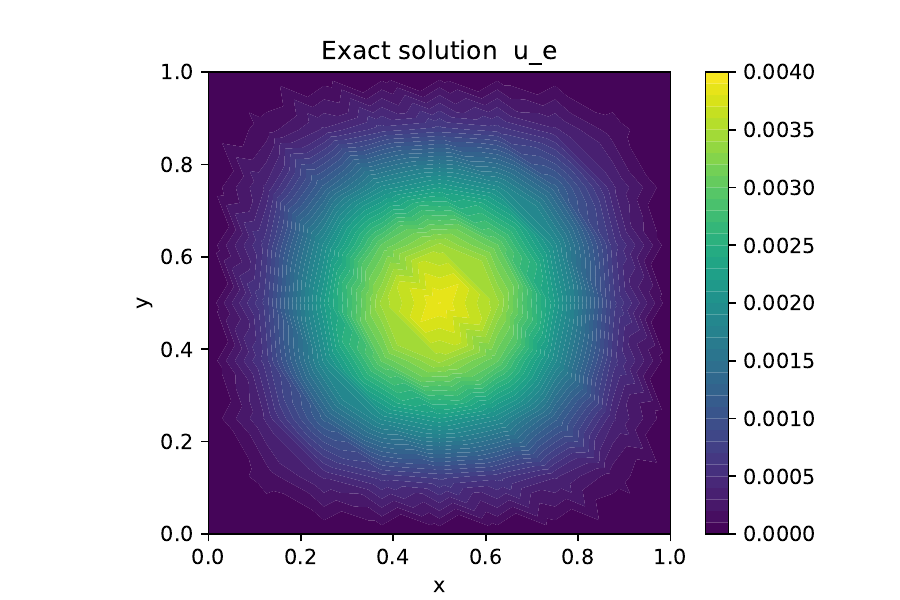}  
		\caption{Exact solution}
		\label{fig:Exct_solu}
	\end{minipage}\hfill
	\begin{minipage}{0.45\textwidth}
		\centering
		\includegraphics[width=\textwidth]{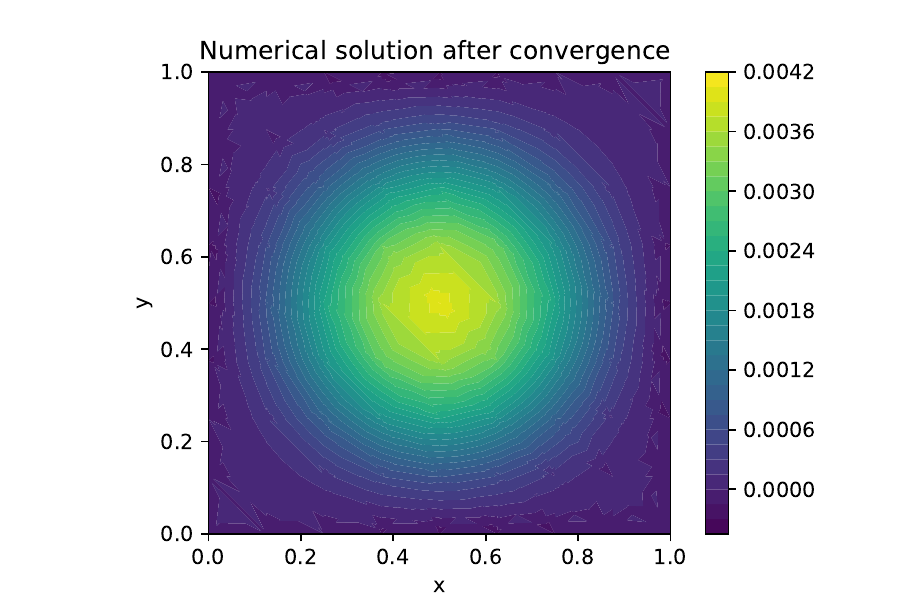}  
		\caption{Numerical solution}
		\label{fig:Numerical_sol}
	\end{minipage}
\end{figure}
\subsection{Test for a priori estimation}\label{priori_estimation}
For error estimation, the errors are defined as:
\begin{eqnarray*}
err_{L2}=\|u-u^{n+1}_h\|_{L^2(\Omega)},\\
err_{H2}=|u-u^{n+1}_h|_{2,\Omega},\text{ where }\\
\|u\|_{L^2(\Omega)}=\left(\int_{\Omega}|u|^2dx\right)^{1/2} \text{ and }\\
|u|_{2,\Omega}=\left(\sum_{i,j=1}^{n}\left\| \dfrac{\partial^2u}{\partial x_i\partial x_j}\right\| ^2_{L^2(\Omega)}\right)^{1/2}.\\
\end{eqnarray*}
The mapping $u \mapsto \|\Delta u\|_{L^2(\Omega)}$ is a norm equivalent to $|\cdot|_{2,\Omega}$ on $H^2_0(\Omega)$, and the latter is the one used for the computations.
The curve in Figure \ref{fig:L2_error} represents the absolute error in the $L^2(\Omega)$ norm, while the curve in Figure \ref{fig:H2_error} represents the absolute error in the $H^2_0(\Omega)$ norm as a function of the mesh size $h$. We tested the algorithm on meshes with $384$ cells and $3840$ degrees of freedom, up to meshes with $1536$ cells and $15360$ degrees of freedom, for $\lambda = 10^4$ and $p = 1$.
\begin{figure}[H]
	\centering
	\begin{minipage}{0.45\textwidth}
		\centering
		\includegraphics[width=\textwidth]{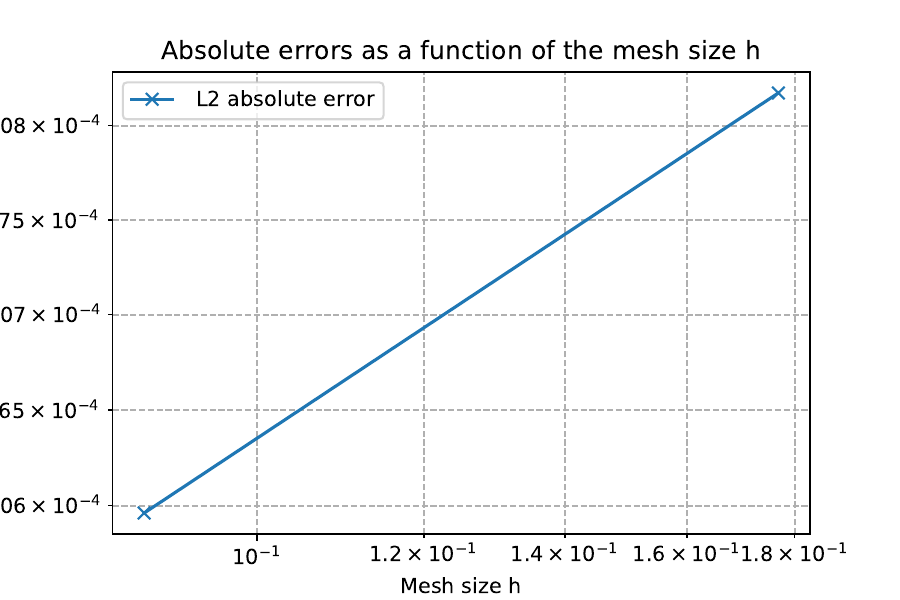} 
		\caption{L2 absolute error curve $err_{L2}$.}
		\label{fig:L2_error}
	\end{minipage}\hfill
	\begin{minipage}{0.45\textwidth}
		\centering
		\includegraphics[width=\textwidth]{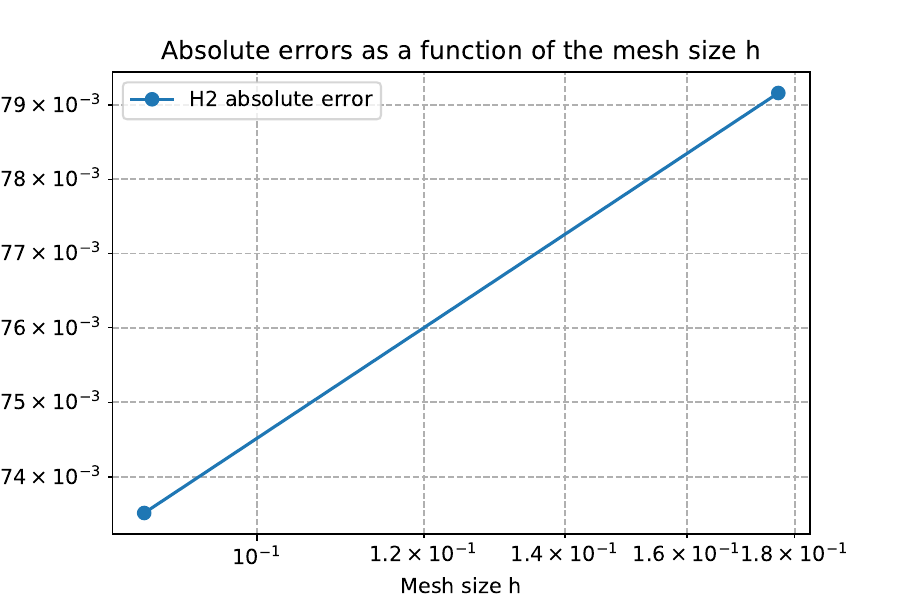}
		\caption{H2 absolute error curve $err_{H2}.$}
		\label{fig:H2_error}
	\end{minipage}
\end{figure}
\brmq We observe that the error decreases as the discretization step tends to zero, with a slope of order $2.05$. This confirms the theoretical results and the second-order convergence established theoretically in the case of HCT meshes.
\ermq
\begin{table}[H]
	\begin{tabular}{|c|c|c|c|c|}
		\hline
		&$p=1$&$p=2$&$p=3$&$p=10$\\
		\hline
		\small$\lambda=10^{-1}$&\small$0.05791626760487744$&\small$0.05791626760385732$&\small$0.05791626760385092$&\small$0.0579162676038548$\\
		\hline
		\small$\lambda=10^{2}$&\small$0.057916267600419936$&\small$0.057916267603858866$&\small$0.05791626760385854$&\small$0.0579162676038548$\\
		\hline
		\small$\lambda=10^{3}$&\small$0.057916267565529256$&\small$0.057916267603424186$&\small$0.05791626760385557$&\small$0.0579162676038548$\\
		\hline
		\small$\lambda=10^{4}$&\small$0.05791626721615341$&\small$0.057916267604456347$&\small$0.05791626760386138$&\small$0.0579162676038548$\\
		\hline
		\small$\lambda=10^{5}$&\small$0.05791626372236575$&\small$0.05791626760443946$&\small$0.057916267603848444$&\small$0.0579162676038548$\\
		\hline
		\small$\lambda=10^{6}$&\small$0.05791622878463318$&\small$0.057916267603997006$&\small$0.057916267603853336$&\small$0.0579162676038548$\\
		\hline
		\small$\lambda=10^{7}$&\small$0.057915879422095966$&\small$0.0579162676001085$&\small$0.05791626760385226$&\small$0.0579162676038548$\\
		\hline
		\small$\lambda=10^{8}$&\small$0.05790517600316137$&\small$0.057916267561317444$&\small$0.05791626760312156$&\small$0.0579162676038548$\\
		\hline
		\small$\lambda=10^{9}$&\small$div$&\small$0.05791626717354878$&\small$0.05791626760378498$&\small$0.0579162676038548$\\
		\hline
		\small$\lambda=10^{10}$&\small$div$&\small$0.057916263295906455$&\small$0.05791626760432593$&\small$0.0579162676038548$\\
		\hline
	\end{tabular}
     \vspace{0.2cm}
	\caption{$err_{H2}=|u-u^{n+1}_h|_{2,\Omega}$}	
	\label{Tab2}
\end{table}
 We denote 'div' when the scheme does not converge after the maximum number of iterations.\\
The table \ref{Tab2} shows that the convergence of the iterative scheme indeed depends on the values of $p$ and $\lambda$. This confirms the theoretical result established in Theorem \ref{thm1}.
\subsection{Test for a posteriori estimation}\label{posteriori_estimation}
In this case, we consider the synthetic solution $u(x,y)=\sin(\pi x)\sin(\pi y)$ defined on the domain $\Omega=(0,1)^2$, which satisfies the boundary conditions.
We implement the iterative fixed-point scheme \ref{P_n}.
Let $u^{n+1}_h$ be the unique solution of \ref{P_n}.
Then, the a posteriori error indicators are locally defined by:
\begin{equation}
\eta^{(D)}_n:=\left(\sum_{K\in \mathcal{T}_h}[\eta^{(D)}_{K,n}]^2\right)^{1/2},
\end{equation}
\begin{equation}
\eta^{(L)}_n:=\left(\sum_{K\in \mathcal{T}_h}[\eta^{(L)}_{K,n}]^2\right)^{1/2},
\end{equation}
where $\eta^{(D)}_{K,n}= \displaystyle h^2_{K}\|f_h-\Delta^2 u^{n+1}_h-\lambda|u^n_h|^{2p}u^n_h\|_{L^2(K)}\\\\+\displaystyle\sum_{e\in \varepsilon(K)}\|\Delta u^{n+1}_h\|_{L^2(e)}
+\displaystyle\frac{1}{2}\sum_{e\in \varepsilon(K)}h^{3/2}_{e}\|[\frac{\partial(\Delta u^{n+1}_h)}{\partial n}]\|_{L^2(e)}$, denotes the discretization error indicator, while  $\eta^{(L)}_{K,n}=\lambda h^2_K|u^n_h-u^{n+1}_h|_{2,K}$ represents the linearization error indicator.
Let $\gamma$ be a positive parameter. We introduce a new linearization stopping criterion used in \cite{Jad} as follows: 
\begin{equation} \label{C11} \eta^{(L)}_n \leqslant \gamma \eta^{(D)}_n \end{equation} and the classical stopping criterion
 \begin{equation} \label{C22} \eta^{(L)}_n \leqslant 10^{-5}. \end{equation} Next, we choose an arbitrary initial approximation $u^0_h$, and we introduce the following iterative algorithm: For $n\in \mathbb{N}$.
\begin{enumerate}
	\item Given $u^n_h$
	\begin{enumerate}
		\item we solve the problem to compute $u^{n+1}_h$.
		\item we calculate $\eta^{(D)}_n$ and $\eta^{(L)}_n$.
	\end{enumerate}
	\item If the linearization error satisfies the stopping criteria \ref{C1} or \ref{C2}, we terminate the iteration loop and proceed to step $(3)$; otherwise, we repeat step $(1)$.
	\item For mesh adaptation: 
	\begin{enumerate}
		\item If $\eta^{(D)}_n$ is below a specified tolerance $\varepsilon$, we stop the algorithm,
		\item otherwise, we perform mesh adaptation, which can be described as follows: Let $ \displaystyle\bar{\eta_n}^{(D)} =\frac{1}{n_t}\sum_{K \in \mathcal{T}_h} \eta^{(D)}_{K,n}$ (where $n_t$ is the number of triangles in the mesh) on a triangle $K$ of the mesh,
		\begin{enumerate}
			\item if $\eta^{(D)}_{K,n}$ is much smaller than $\bar{\eta_n}^{(D)}$, we coarsen the mesh around $K$.
			\item if $\eta^{(D)}_{K,n}$ is much larger than $\bar{\eta_n}^{(D)}$, we refine the mesh around $K$.
		\end{enumerate}
		\item Subsequently, we return to step $(1)$.
	\end{enumerate}
\end{enumerate}
Using the criteria based on the local error indicators  \ref{C11}  and   \ref{C22}, we achieved convergence, and figure \ref{fig:Numerical} displays the numerical solution after convergence compared to the exact solution presented in figure \ref{fig:Exct}. This result is very satisfactory for $\lambda=10^6$ and $p=1$, highlighting once again the effectiveness of local error indicators in the numerical resolution of PDEs. Moreover, figure \ref{fig:error} presents the curve describing the evolution of the two errors and the two error indicators as a function of the discretization step size. Table \ref{Tab5} shows the repartition of $err_{H2}$ and $\eta^{(D)}_{n}$ values as a function of $h$.
\begin{figure}[H]
	\centering
	\begin{minipage}{0.45\textwidth}
		\centering
		\includegraphics[width=\textwidth]{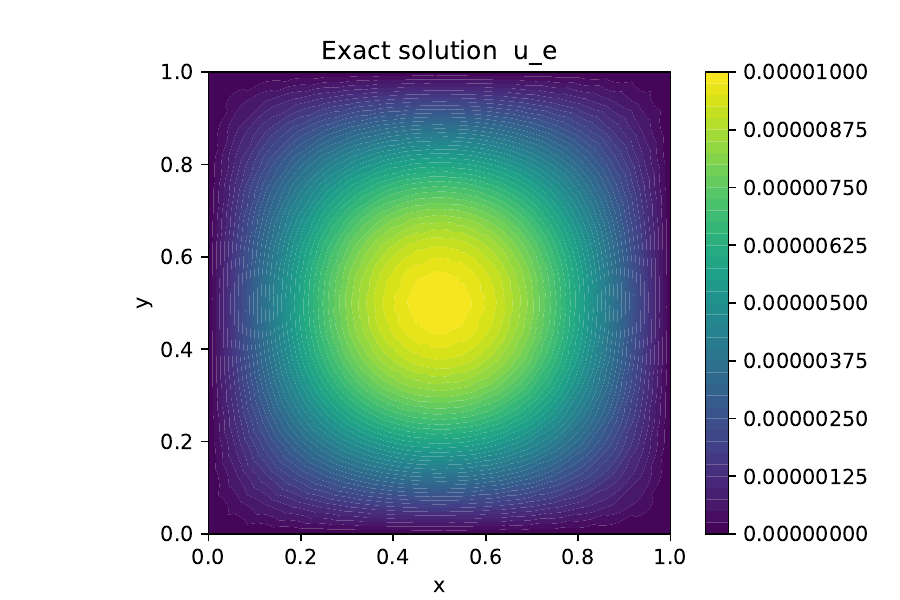}  
		\caption{Exact solution}
		\label{fig:Exct}
	\end{minipage}\hfill
	\begin{minipage}{0.45\textwidth}
		\centering
		\includegraphics[width=\textwidth]{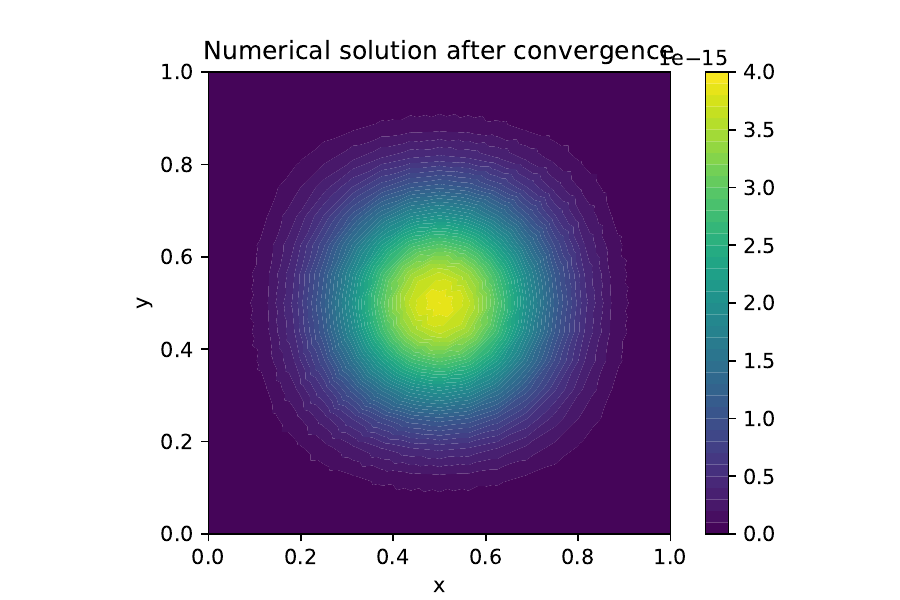}  
		\caption{Numerical solution}
		\label{fig:Numerical}
	\end{minipage}
\end{figure}
\begin{figure}[H]
	\includegraphics[width=0.5\textwidth]{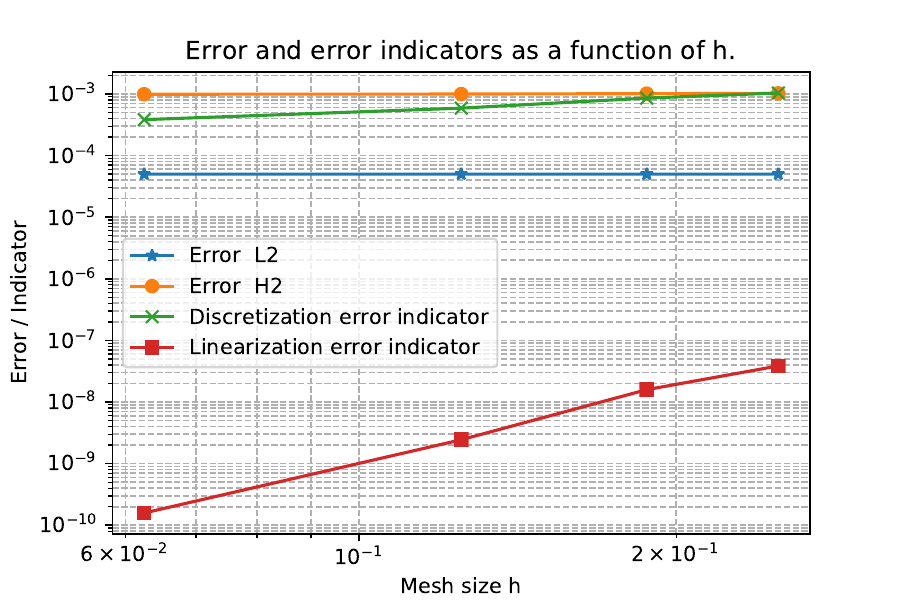} 
	\caption{Curve Error/ Indicator}
	\label{fig:error}
\end{figure}
\begin{table}[H]
	\begin{tabular}{|c|c|c|c|c|}
		\hline
		\small$h$&\small$err_{H2}=|u-u^{n+1}_h|_{2,\Omega}$&\small$\eta^{(D)}_{n}$&\small$\eta^{(L)}_{n}$&\small$
		\dfrac{\left([\eta^{(D)}_{n}]^2+[\eta^{(L)}_{n}]^2\right)^{\frac{1}{2}}}{err_{H2}}$\\
		\hline
		\small$0.25$&\small$0.0010227068705228558$&\small$0.0010359541676476713$&\small$4.7967065729477115e-08$&\small$1.0129531722306075$\\
		\hline
		\small$0.1875$&\small$0.0010103485949389303$&\small$0.0008563221777400873$&\small$3.145699344479001e-08$&\small$0.8475512141130194$\\
		\hline
		\small$0.125$&\small$0.000996339947451816$&\small$0.0005881536238968676$&\small$1.2210650637519503e-08$&\small$0.590314205033984$\\
		\hline
		\small$0.0625$&\small$0.0009893373045555425$&\small$0.00038291768460119846$&\small$3.159212355349324e-09$&\small$0.3870446235586514$\\
		\hline
	\end{tabular}
\vspace{0.2cm}
	\caption{Repartition of $err_{H2}$ values according $\eta^{(D)}_{n}$ and $\eta^{(L)}_{n}$}
	\label{Tab5}
\end{table}
\vspace{0.5cm}
\brmq
We observe that the results obtained in Figure \ref{fig:error} and Table \ref{Tab5} confirm the theoretical results established, demonstrating the reliability and efficiency of the family of local indicators.
\ermq
\subsection{Mesh adaptation}\label{adapt}
In this subsection, we illustrate the evolution of the mesh adaptation. Starting from an initial HCT mesh (Figure \ref{fig:Initial_mesh_HCT}), Figures \ref{fig:Adapted_mesh1}, \ref{fig:Adapted_mesh2}, and \ref{fig:Adapted_mesh3} show the progression of the mesh over the iterations. A significant concentration of mesh adaptation is observed in the region where the solution is non-zero.
\begin{figure}[H]
	\centering
	\begin{minipage}{0.45\textwidth}
		\centering
		\includegraphics[width=\textwidth]{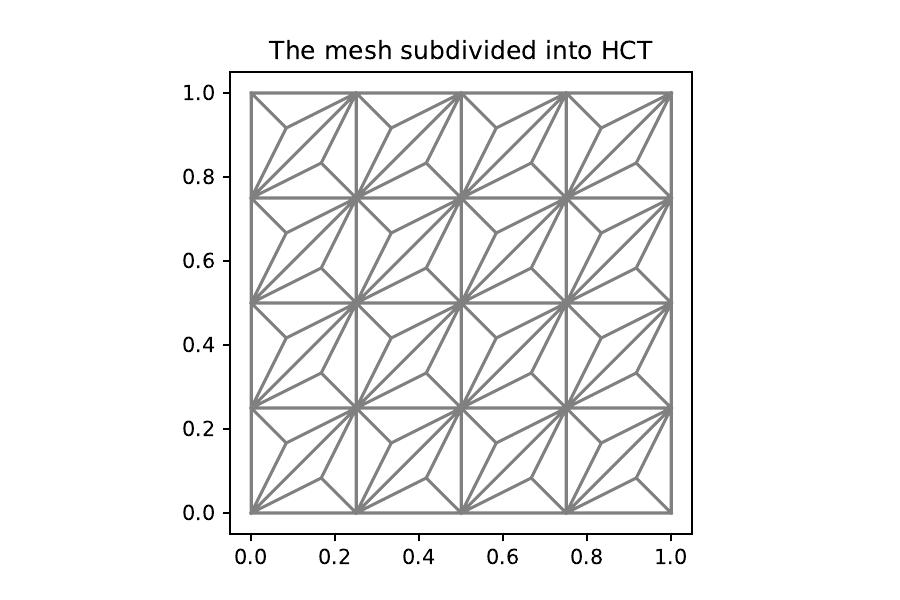}
		\caption{Initial mesh HCT: 96 cells}
		\label{fig:Initial_mesh_HCT}
	\end{minipage}\hfill
	\begin{minipage}{0.45\textwidth}
		\centering
		\includegraphics[width=\textwidth]{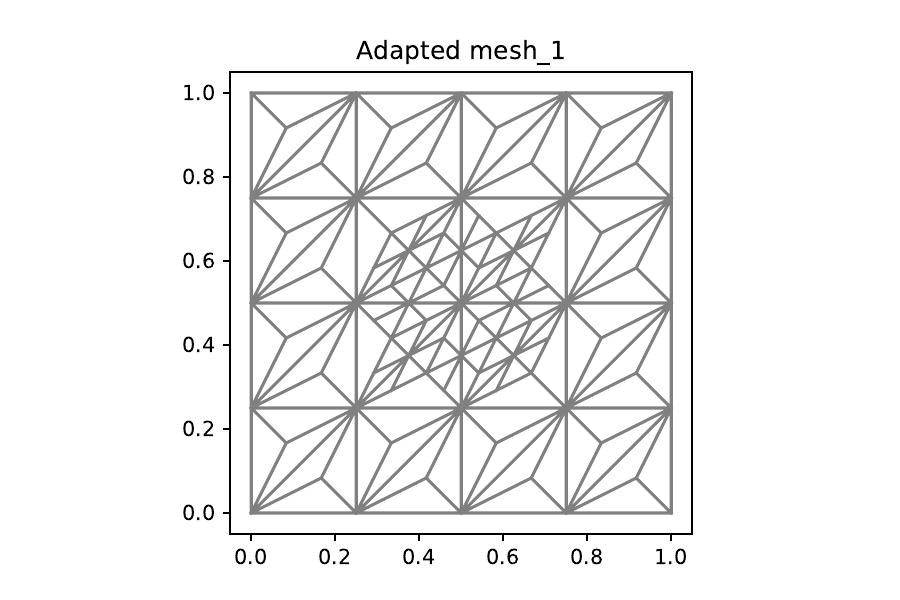}
		\caption{First adapted mesh: 144 cells}
		\label{fig:Adapted_mesh1}
	\end{minipage}
\end{figure}
\begin{figure}[H]
	\centering
	\begin{minipage}{0.45\textwidth}
		\centering
		\includegraphics[width=\textwidth]{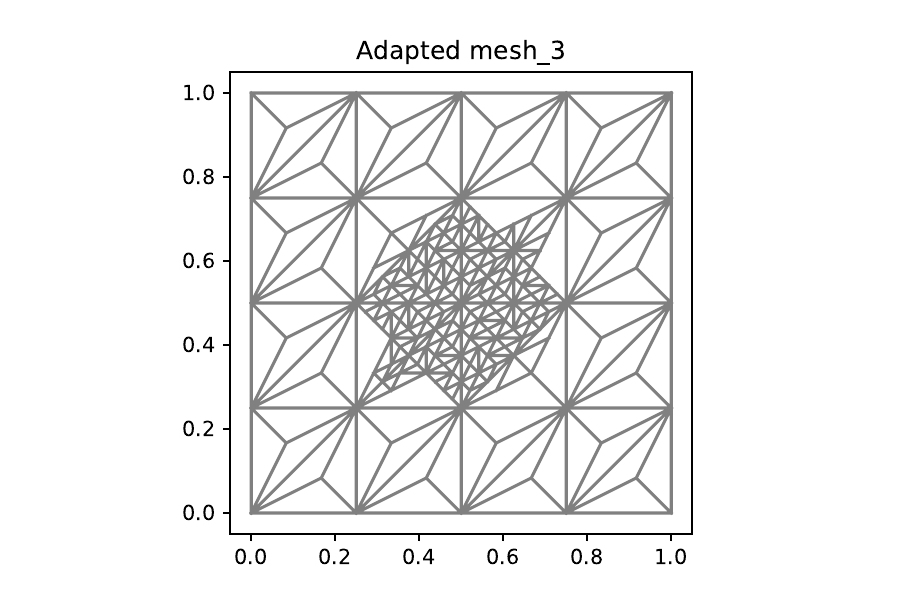}
		\caption{Second adapted mesh: 300 cells }
		\label{fig:Adapted_mesh2}
	\end{minipage}\hfill
	\begin{minipage}{0.45\textwidth}
		\centering
		\includegraphics[width=\textwidth]{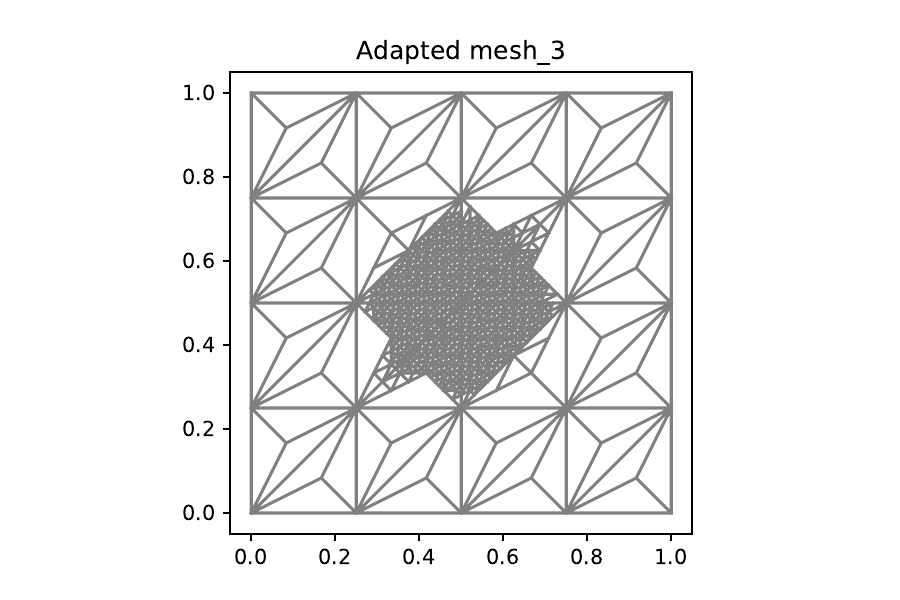}
		\caption{Third adapted mesh: 894 cells}
		\label{fig:Adapted_mesh3}
	\end{minipage}
\end{figure}
\subsection{Discussions}
\label{resultats}
In this subsection, we present the observations from the numerical tests and provide a discussion around them. 
The initial mesh (Figure \ref{fig:Initial_mesh}) is a triangulation of the domain $\Omega=]0;1[^2$ and the figure (Figure \ref{fig:HCT_mesh}) is the HCT mesh generated from this initial triangulation, while the adapted meshes(Figures \ref{fig:Adapted_mesh1} to \ref{fig:Adapted_mesh3}) show a strong concentration in the areas where the solution presents significant variations. The obtained numerical solution (Figures \ref{fig:Numerical_sol} and \ref{fig:Numerical}) is visually close to the exact solution (Figures \ref{fig:Exct_solu} and \ref{fig:Exct}), with a rapid convergence after a few iterations. The error curves (Figures \ref{fig:L2_error} and \ref{fig:H2_error}) show that the errors decrease monotonically with the mesh step $h$, and that of the figure (Figure\ref{fig:error}) presents the correlation between the error indicators and the errors. The analysis of the table \ref{Tab2} reveals 
that the values of the error $|u-u_h^{n+1}|_{2,\Omega}$ show a notable dependence on the parameters $p$ and $\lambda$ and that of the table \ref{Tab5} presents the distribution of the error indicators $ \eta^{(D)}_n $ and of the errors $ |u-u^{n+1}_h|_{2,\Omega}$ which shows a strong decrease with the mesh refinement.
  
The numerical results presented in this study highlight several key aspects that demonstrate their importance, both theoretically and practically. We can discuss:
\begin{enumerate}
	\item Validation of theoretical predictions:
	The results confirm the theoretical predictions on second-order convergence for HCT meshes. This validation is essential because it ensures that the adopted methodology (HCT elements, error indicators, iterative scheme) is consistent with the established theoretical properties.
	\item Effectiveness of a posteriori error estimators:
	The a posteriori error estimators, $\eta^{(D)}_n$ and $\eta^{(L)}_n$, were found to be crucial in establishing a stopping criterion based on the balance between discretization and linearization errors that proved robust in handling iterative convergence. They allowed identifying regions requiring mesh refinement. The efficient use of these indicators reduces computational costs while maintaining high accuracy, which is essential for complex nonlinear problems.
	\item Impact of mesh adaptation:
	Mesh adaptation shows a concentration of cells in areas where the solution presents strong variations. This behavior is illustrated by the evolution of the mesh figures. The approach guarantees an optimal allocation of computational resources, making the method competitive even for complex problems.
\end{enumerate}
\section{Summary }\label{Conclusion}
We analyzed a posteriori residual error estimators of a fixed point iterative algorithm for the plate bending problem perturbed by a nonlinear term.
This analysis allowed us to identify two sources of error: the discretization error and the linearization error.
The local error indicators obtained are both reliable and efficient. The numerical results obtained fully corroborate the theoretical predictions, thereby validating the coherence and relevance of the proposed model. They demonstrate the efficiency, accuracy, and robustness of the developed method. Moreover, they provide strong experimental validation of the theoretical concepts, positioning this approach as a promising solution for complex nonlinear modeling problems. Many issues remain to be addressed in this area, let us mention other methods, namely, \emph{ADINI's element} \cite{Aniso} or \emph{nonconforming VIRTUAL element method}  for plate bending problems recently developed in \cite{MEV}.
\section{Acknowledgments}
We are thankful to the editor and the anonymous reviewers for many valuable suggestions to improve this paper.
\newcommand{\noopsort}[1]{}

\end{normalsize}



\end{document}